\documentclass{amsart}

\usepackage{hyperref}
\hypersetup{
    colorlinks,
    linkcolor={red!50!black},
    citecolor={blue!50!black},
    urlcolor={blue!80!black}
}
\usepackage{amsmath, amsthm, amssymb, mathrsfs, verbatim, mathtools, tikz-cd, xcolor}

\usepackage[style=alphabetic]{biblatex}
\newcommand{\st}{~|~}
\newcommand{\mst}{~\middle|~}

\newcommand{\R}{\mathbb{R}}
\newcommand{\C}{\mathbb{C}}
\newcommand{\Z}{\mathbb{Z}}
\newcommand{\N}{\mathbb{N}}

\renewcommand{\P}{\mathbb{P}}

\renewcommand{\phi}{\varphi}
\renewcommand{\epsilon}{\varepsilon}
\renewcommand{\emptyset}{\varnothing}
\renewcommand{\setminus}{\smallsetminus}
\newcommand{\e}{\epsilon}
\newcommand{\dee}{\partial}

\newcommand{\Hom}{\mathrm{Hom}}

\DeclareMathOperator{\im}{\operatorname{im}}

\newcommand{\Sym}{\mathrm{Sym}}

\DeclareMathOperator{\Hol}{\mathrm{Hol}}

\newcommand{\1}{^{-1}}

\newcommand{\la}{\langle}
\newcommand{\ra}{\rangle}

\renewcommand{\Tilde}{\widetilde}

\renewcommand{\Im}{\mathrm{Im}}

\newtheorem{mainthm}{Theorem}

\newtheorem{maincor}[mainthm]{Corollary}

\newtheorem{thm}{Theorem}
\numberwithin{thm}{section}
\newtheorem{prop}[thm]{Proposition}
\newtheorem{lem}[thm]{Lemma}
\newtheorem{cor}[thm]{Corollary}
\newtheorem*{claim}{Claim}

\theoremstyle{definition}

\newtheorem{defn}[thm]{Definition}
\newtheorem{ex}[thm]{Example}
\AtBeginEnvironment{ex}{%
    \pushQED{\qed}%
}
\AtEndEnvironment{ex}{\popQED}
\newtheorem{rem}[thm]{Remark}

\begin{document}
\title[Convex iso-Delaunay regions]{Convex iso-Delaunay regions in strata of\\ translation surfaces}
\author{Sam Freedman}
\author{Bradley Zykoski}
\address{University of Chicago}
\email{sfreedman67@uchicago.edu}
\address{Northwestern University}
\email{zykoski@northwestern.edu}
\date{September 16, 2025}
\begin{abstract}
Strata of translation surfaces are covered by the closures of finitely many iso-Delaunay regions: open subspaces parametrizing surfaces whose Delaunay triangulations are combinatorially equivalent.
We prove that the iso-Delaunay regions for triangulations with involutions called triangle matchings are convex with respect to the angle coordinates of the triangles.
As a corollary, we show that all iso-Delaunay regions in hyperelliptic stratum components are convex.
This work generalizes two directions in the theory of flat surfaces.
\end{abstract}

\maketitle

\setcounter{tocdepth}{1}
\tableofcontents

\section{Introduction}
For $g \ge 1$, let $S_g$ be a closed surface of genus $g$, and let $Z = \{x_1, \dots, x_n \} \subset S_g$ be a set of $n$ points.
Every integer partition $\kappa = (k_1,\dots,k_n)$ of $2g-2$ determines a \emph{Teichm\"uller stratum of translation surfaces} $\Omega \mathcal{T}_{g, n}(\kappa)$.
This is the space of  of \emph{Teichm\"uller equivalence classes} of translation surfaces $(X, \omega)$, where $X$ is a Riemann surface of genus $g$ and $\omega \in H^{1,0}(X)$ is a nontrivial abelian differential, along with a \emph{marking} homeomorphism $f: (S_g, Z) \to (X, \omega)$ such that $\omega$ has a zero of order $k_i$ at the point $f(x_i)$.
We will work with the complex projectivization $\P \Omega \mathcal{T}_{g, n}(\kappa)$ of this space because the objects we study are invariant under the $\C^*$-action on $\Omega \mathcal{T}_{g, n}(\kappa)$.

Every translation surface may be decomposed into a collection of polygons whose opposite sides have been identified by Euclidean translations.
A generic translation surface $(X, \omega)$ admits a canonical such decomposition, where every polygon is a triangle, called the \emph{Delaunay triangulation} of $(X, \omega)$.
This triangulation is dual to the classical Voronoi decomposition of $(X, \omega)$ with respect to the zeros of $\omega$.
The strata $\P \Omega \mathcal{T}_{g, n}(\kappa)$ naturally decompose as a countable union of \textit{iso-Delaunay regions} $\Delta_{\mathcal{D}}(\tau)$, which are subspaces of translation surfaces whose Delaunay triangulations are combinatorially equivalent to $\tau$.
Masur--Smillie first used these regions in \cite{MasurSmillie} to obtain quantitative bounds on geometric quantities such as the lengths of saddle connections and volumes of strata to analyze degenerations of flat structures.

This work concerns the geometry of iso-Delaunay regions in strata of translation surfaces, and it may be viewed as a common generalization of the study of iso-Delaunay regions in two other contexts.
\begin{enumerate}
    \item\label{BowmanVeech}
    Bowman \cite{Bowman} and Veech \cite{VeechBicuspid} proved that Teichm\"uller disks intersect iso-Delaunay regions in convex hyperbolic polygons, which they used to algorithmically compute topological statistics of Teichm\"uller curves.
    \item\label{RivinVeech} Rivin \cite{Rivin} and Veech \cite{VeechDelaunay} proved that the iso-Delaunay regions in moduli spaces of complex-affine structures are convex in angle coordinates, yielding results on the local geometry of those moduli spaces.
\end{enumerate}
The rigid nature of Teichm\"uller disks in (\ref{BowmanVeech}) and the flexible nature of complex-affine structures in (\ref{RivinVeech}) are necessary for these results.
Strata of translation surfaces enjoy neither of these features,
and the structure of iso-Delaunay regions in this setting has remained mysterious.
In particular, no iso-Delaunay regions in $\P\Omega \mathcal T_{g, n}(\kappa)$ for $g \ge 2$ were known to be contractible prior to this work.

In Theorem~\ref{thm:mainA}, we prove that the iso-Delaunay regions for triangulations admitting involutions called \emph{triangle matchings} (Definition~\ref{defn:triangle_matching}) are convex with respect to \emph{angle coordinates} (Definition~\ref{defn:angle_map}), which parametrize the Euclidean angles in the triangles.

\begin{mainthm}[Theorem~\ref{thm:main_theorem_precise}]\label{thm:mainA}
Let $\tau$ be a topological triangulation of $(S_g, Z)$ such that $\Delta_{\mathcal D}(\tau)$ is nonempty.
If $\tau$ admits a triangle matching, then $\Delta_{\mathcal{D}}(\tau)$ is convex with respect to angle coordinates.
In particular, $\Delta_{\mathcal{D}}(\tau)$ is contractible.
\end{mainthm}

Furthermore, in \S \ref{sec:examples} we provide explicit constructions of convex iso-Delaunay regions, from which we obtain the following corollaries.

\begin{maincor}[Lemma~\ref{lem:hyperelliptic_delaunay}]\label{cor:hyperelliptic}
All iso-Delaunay regions in hyperelliptic stratum components are convex with respect to angle coordinates.
\end{maincor}

\begin{maincor}[Corollary~\ref{cor:arboreal_in_every_stratum}]\label{cor:every_stratum}
    All components of the projectivized strata $\P\Omega \mathcal{T}_{g, n}(\kappa)$ contain a contractible iso-Delaunay region.
\end{maincor}

\subsection*{Running example: the square L}

\begin{figure}
    \centering
    \includegraphics[width = 0.5\linewidth]{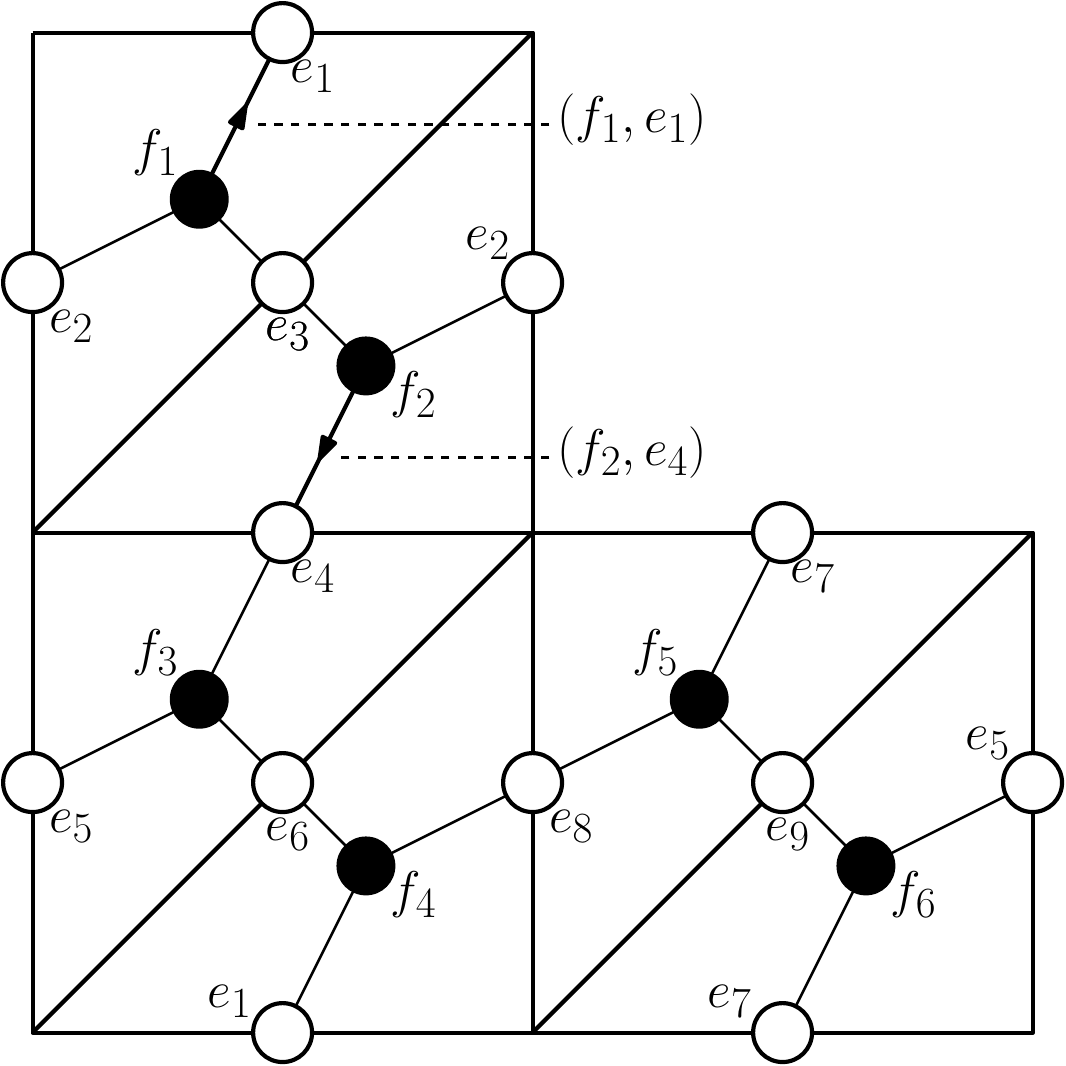}
    \caption{The trivalent ribbon graph $\Gamma(\tau)$ dual to a topological triangulation $\tau$ on the square L.
    Two half-edges $(f_1, e_1)$ and $(f_2, e_4)$ have been labeled with our orientation convention.}
    \label{fig:square_l_combinatorial}
\end{figure}

Consider the L-shaped translation surface $(X, \omega) \in \Omega \mathcal T_{2, 1}(2)$ depicted in Figure~\ref{fig:square_l_combinatorial}.
Figure~\ref{fig:square_l_combinatorial} also depicts the \emph{standard triangulation} $\tau$ of $(X, \omega)$ (Definition~\ref{defn:standard_triangulation}), along with its dual trivalent ribbon graph $\Gamma(\tau)$ (Definition~\ref{defn:dual_graph}).
The square L exhibits the fundamental concepts in this paper:
\begin{itemize}
    \item The ribbon graph $\Gamma(\tau)$ admits a \emph{triangle matching} (Definition~\ref{defn:triangle_matching}), as in Figure~\ref{fig:square_l}.
    \item The image of the \emph{iso-Delaunay region} $\Delta_{\mathcal D}(\tau)$ (Definition~\ref{defn:isodelaunay_region}) under the \emph{angle map} $\Theta$ (Definition~\ref{defn:angle_map}) is convex by Theorem~\ref{thm:main_theorem_precise}.
    \item The triangulation $\tau$ is \emph{compatible} with a hyperelliptic involution $\iota: S_g \to S_g$ (Definition~\ref{hyperelliptic defn}). 
    \item The translation surface $(X, \omega)$ is a \emph{geometrically simple origami} (Definition~\ref{defn:origami}), with $h = (12)$ and $v = (13)$, that is \emph{arboreal} (Definition~\ref{defn:arboreal}).
\end{itemize}

\subsection*{Outline}
An iso-Delaunay region is defined by \emph{holonomy equations} and \emph{Delaunay inequalities} on triangle angles; a crucial difficulty is that while the latter are affine-linear, the former are transcendental and opaque.
Our key insight is that when $\tau$ admits a triangle matching $\iota: H \to H$ on the \emph{half-edges} $H$ of $\Gamma(\tau)$, these constraints reduces to an affine-linear system.

We first develop in \S\S\ref{sec:translation_surfaces}-\ref{sec:ribbon_graphs} the necessary background on translation surfaces and trivalent ribbon graphs, respectively.
In \S\ref{sec:complex_affine}, we define the \emph{angle coordinate map} $\Theta$ on the \emph{iso-triangulable region} $\Delta = \Delta(\tau)$, and we show that its image $\Theta(\Delta)$ is the locus of trivial holonomy in \emph{angle space} $\mathcal{A}$ (Lemma~\ref{lem:image_of_angle_map}).
Next, in \S\ref{sec:triangle_matchings} we introduce triangle matchings, identify the affine-linear $\iota$-invariant subspace $\mathcal A^\iota \subset \mathcal{A}$, and prove in \S\ref{sec:proof_of_main_theorem} that $\Theta(\Delta) = \mathcal A^\iota$ (Theorem~\ref{thm:main_theorem_precise}). The affine-linearity of the Delaunay inequalities in angle coordinates implies the convexity of $\Delta_{\mathcal D}$.
Finally, in \S\ref{sec:examples} we construct examples of convex iso-Delaunay regions: those containing hyperelliptic surfaces (\S\ref{subsec:hyperelliptic}), sums of surfaces (\S\ref{subsec:connected_sum}) and \emph{arboreal origamis} (\S\ref{subsec:arboreal}).
In \S\ref{sec:future_directions}, we compare our results with the study $L^\infty$-Delaunay triangulations and the topology of strata.

\subsection*{Acknowledgements}
The authors thank Aaron Calderon for the suggestion to consider hyperelliptic components and Alex Eskin, Simion Filip, and Eduard Looijenga for helpful discussions.
The authors thank Curt McMullen for comments on a previous draft.
The first author thanks Jeremy Kahn and the second author thanks Alex Wright for suggesting the study of iso-Delaunay regions that led to this work.
The first author thanks Vincent Delecroix and Julien R\"uth for support with \texttt{sage-flatsurf}.
The second author was supported by the NSF grant DMS-2136217.

\newpage
\section{Triangulations of translation surfaces}\label{sec:translation_surfaces}
In this section, we establish our notation and recall basic material on translation surfaces and their triangulations.
In particular, we define the iso-triangulable region $\Delta(\tau)$ and iso-Delaunay region $\Delta_{\mathcal D}(\tau)$ associated to a triangulation $\tau$ of a surface.
A basic reference on translation surfaces is \cite{AthreyaMasur}.

\begin{defn}\label{defn:translation_surface}
A \emph{marked translation surface} $(X, \omega)$ \emph{with zero multiplicities} $\kappa$ \emph{on} $(S_g, Z)$ is a translation surface along with an orientation-preserving diffeomorphism (a \emph{marking}) $\phi: S_g \to X$ such that $\omega$ vanishes to order $k_j$ at each $\phi(x_j)$ for $x_j \in Z$. When $k_j = 0$, the differential $\omega$ is nonvanishing at $\phi(x_j)$, and we say that $\phi(x_j)$ is a \emph{marked point} of $\omega$; otherwise, we say that $\phi(x_j)$ is a \emph{singularity} of $\omega$.
\end{defn}

\begin{defn}\label{defn:stratum}
The \emph{Teichm\"uller stratum} $\Omega \mathcal T_{g,n}(\kappa)$ is the space of equivalence classes $M = [(X, \omega)]$ of marked translation surfaces $(X, \omega)$ on $(S_g, Z)$ under the following equivalence relation.
Two marked translation surfaces $(X,\omega)$ and $(X', \omega')$ (with marking diffeomorphisms $\phi$ and $\phi'$) are \emph{Teichm\"uller equivalent} if there is a biholomorphism $\psi: X \to X'$ with $\psi^* \omega' = \omega$ such that $\psi \circ \phi$ is isotopic rel $Z$ to $\phi'$.
\end{defn}

\begin{rem}
Given a marked translation surface $(X, \omega)$ on $(S_g, Z)$ with marking $\phi$, we obtain a Riemann surface structure on $S_g$ by precomposing coordinate charts $X \supset U \to \C$ with the marking.
In this way, we may regard $X$ as a Riemann surface structure on $S_g$ itself, and $\omega$ as a differential form on $S_g$ with zeroes in $Z$.
We therefore assume without loss of generality that $(S_g, Z)$ is the base topological surface of every marked translation surface that we consider.
\end{rem}

For each $\gamma \in H_1(S_g, Z; \Z)$, we have a \emph{period map} $\Phi_\gamma: \Omega \mathcal T_{g,n}(\kappa) \to \C$ given by $[(X, \omega)] \mapsto \int_\gamma \omega$, which is well-defined because the class $\gamma$ is invariant under isotopy rel $Z$.
The topology on the space $\Omega \mathcal T_{g,n}(\kappa)$ is the weakest topology that makes the maps $\Phi_\gamma$ continuous.
Given a basis $\gamma_1, \dots, \gamma_{2g + n - 1}$ of $H_1(S_g, Z; \Z)$, it is a basic fact of Teichm\"uller theory that the map $\Phi = \prod_{j=1}^{2g + n - 1} \Phi_{\gamma_j}: \Omega \mathcal T_{g,n}(\kappa) \to \C^{2g + n -1}$ is a local homeomorphism, called the \emph{period coordinate map} determined by the basis $\gamma_1, \dots, \gamma_{2g + n - 1}$.

\begin{defn}
The group $\C^*$ acts on $\Omega \mathcal T_{g,n}(\kappa)$ by rescaling the differential: for $M = [(X, \omega)] \in \Omega \mathcal T_{g,n}(\kappa)$, we have $\lambda M = [(M, \lambda \omega)]$ for every $\lambda \in \C^*$.
We define $\P\Omega \mathcal T_{g,n}(\kappa)$ as the quotient of $\Omega \mathcal T_{g,n}(\kappa)$ by this action.
\end{defn}

Note that there is a natural injective map $\mathbb{P} \Omega \mathcal{T}_{g, n}(\kappa) \to \mathcal{T}_{g, n}$ by forgetting the abelian differential.

\begin{rem}
Given a translation surface $(X,\omega)$ on $(S_g, Z)$ with multiplicities $\kappa$, there is a flat incomplete Riemannian metric on $S_g \setminus Z$ induced by coordinate charts $\varphi: U \to \C$ that satisfy $\varphi^*dz = \omega$.
This is the \emph{Euclidean metric} on $(X, \omega)$.
At each point of $x_j \in Z$, there is a homeomorphism of a neighborhood of $x_j$ to a neighborhood of the vertex of a Euclidean cone with cone angle $2\pi (k_j + 1)$ that is an isometry away from $x_j$.
We may therefore speak of, for example, line segments and Euclidean polygons on $(X, \omega)$.
\end{rem}

\begin{defn}
A \emph{topological triangulation} $\tau$ of $(S_g, Z)$ is a CW decomposition of $S_g$ such that $Z$ is the set of 0-cells, and the resulting complex is a simplicial complex. We will write $V(\tau) = Z, E(\tau), F(\tau)$ for the sets of vertices, edges, and faces, respectively, of $\tau$.
\end{defn}

\begin{defn}
A \emph{geometric triangulation} of a translation surface $(X,\omega)$ on $(S_g, Z)$ is a topological triangulation $\tau$ such that each $f \in F(\tau)$ is a Euclidean triangle on $(X, \omega)$.
\end{defn}

\begin{defn}\label{isotriangulable defn}
Let $\tau$ be a topological triangulation.
The \emph{iso-triangulable region} $\Tilde\Delta(\tau) \subset \Omega \mathcal T_{g,n}(\kappa)$ is the set of all Teichm\"uller equivalence classes $M = [(X, \omega)] \in \Omega \mathcal T_{g,n}(\kappa)$ where $(X,\omega)$ a marked translation surface for which $\tau$ is a geometric triangulation.
As $\Tilde \Delta(\tau)$ is invariant under the action of $\C^*$, there is a well-defined projectivized \emph{iso-triangulable region} $\Delta(\tau) \subset \P \Omega \mathcal T_{g,n}(\kappa)$.
\end{defn}

We now discuss Delaunay triangulations, adapting the presentation in \cite{BobenkoSpringborn}.
\begin{figure}
    \centering
    \includegraphics[width=0.5\linewidth]{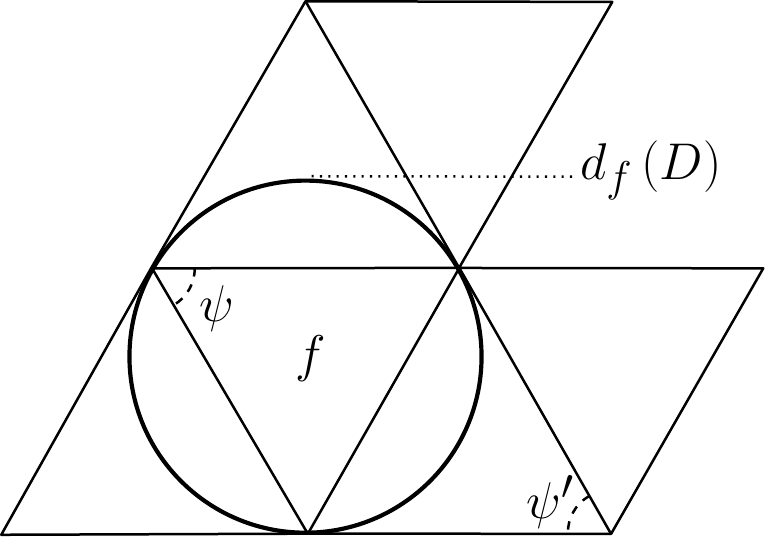}
    \caption{A Delaunay triangulation of an equilaterally-tiled translation surface. Corollary~\ref{cor:delaunay_criterion} implies that $\psi + \psi' < \pi$.}
    \label{fig:delaunay}
\end{figure}

\begin{defn}
Let $(X, \omega)$ be a translation surface.
An \emph{immersed empty disk} is a continuous map $d : \overline{D} \to X$, where $D \subset \mathbb{R}^2$ is an open Euclidean disk such that $d(D) \cap Z = \emptyset$, and $d|_D$ is an isometric immersion with respect to the Euclidean metric on $(X, \omega)$.
\end{defn}

\begin{defn}\label{defn:delaunay_triangulation}
A \emph{Delaunay triangulation} $\tau$ of $(X,\omega)$ is a geometric triangulation such that for each face $f \in F(\tau)$, there is an immersed empty disk $d_f : \overline{D} \to X$ such that there are three points $p_1, p_2, p_3 \in d_f^{-1}(Z)$ such that $f$ is the image of the Euclidean triangle with vertices $\{p_1, p_2, p_3\}$. See Figure~\ref{fig:delaunay}.

There might be more than three points in $d_f^{-1}(Z)$.
If for every $f \in F(\tau)$ there are exactly three points in $d_f^{-1}(Z)$, we say that the Delaunay triangulation $\tau$ is \emph{nondegenerate}.
\end{defn}

\begin{rem}\label{rem:uniqueness}
If $(X, \omega)$ admits a nondegenerate Delaunay triangulation $\tau$, then $\tau$ is the unique Delaunay triangulation of $(X, \omega)$. Hence, in this case, $\tau$ is uniquely determined by the metric structure of $(X, \omega)$.
See the discussion in \cite[\S 2]{BobenkoSpringborn}.
\end{rem}

\begin{defn}\label{defn:isodelaunay_region}
Let $\tau$ be a topological triangulation of $(S_g, Z)$.
The \emph{iso-Delaunay region} $\Delta_{\mathcal D}(\tau) \subset \Delta(\tau)$ is the subset of all Teichm\"uller equivalence classes $M = [(X, \omega)] \in \Delta(\tau)$ where $(X, \omega)$ is a marked translation surface for which $\tau$ is a nondegenerate Delaunay triangulation.
\end{defn}

Note that both $\Delta(\tau)$ and $\Delta_{\mathcal D}(\tau)$ are open subsets of $\mathbb{P}\Omega\mathcal{T}_{g, n}(\kappa)$.

\section{Ribbon graphs}\label{sec:ribbon_graphs}
In \S\ref{subsec:combinatorial_definitions}, we establish our notation for trivalent ribbon graphs, which will be the basic combinatorial framework for this paper.
In \S\ref{subsec:combinatorial_holonomy}, we describe the combinatorial analogue to the holonomy of a complex-affine structure, which will be used as a bridge between the combinatorial and geometric aspects of triangulations of translation surfaces.

\subsection{Basic definitions}\label{subsec:combinatorial_definitions}
\begin{defn}\label{defn:trivalent_ribbon_graph}
    Let $E$ and $F$ be finite sets.
    A graph $\Gamma$ with vertex set $E \cup F$ and edge set $H$ is a \emph{trivalent ribbon graph} if
\begin{enumerate}
    \item The vertices $e \in E$ have valence 2, and the vertices $f \in F$ have valence 3.
    \item There are \emph{endpoint maps} $\e:H \to E$ and $\phi:H \to F$ such $e = \e(h)$ and $f = \phi(h)$ are the vertices of $h$. We write $h = (f, e)$.
    \item There is a fixed-point-free action $\Z/3\Z \curvearrowright H_f$ on each fiber $H_f = \phi\1(f)$ and therefore an action $\Z/3\Z  \curvearrowright H$.
    Given $n \in \Z/3\Z$ and $h = (f, e) \in H_f$, we write $(f, e + n)$ for the image of $h$ under the action of $n$.
\end{enumerate}
We refer to the edges $(f, e)$ of $\Gamma$ as \emph{half-edges}.
\end{defn}

\begin{defn}
For each half-edge $(f, e) \in H$, we define the \emph{combinatorial angle} as the set
\[
(f; e, e+1) \coloneqq \{(f, e), (f, e + 1)\}.
\]
Let $A$ denote the set of all combinatorial angles endowed with the $(\Z/3\Z)$-action $(f; e, e+1) + n \coloneqq (f; e + n, e + n + 1)$.
We refer to $\Z \la A \ra$ as the \emph{angle group} $\Z\la A \ra$, and we write $(f; e + 1, e) \coloneqq -(f; e, e+1) \in \Z\la A \ra$.
\end{defn}

\begin{figure}
    \centering
    \includegraphics[width=0.6\linewidth]{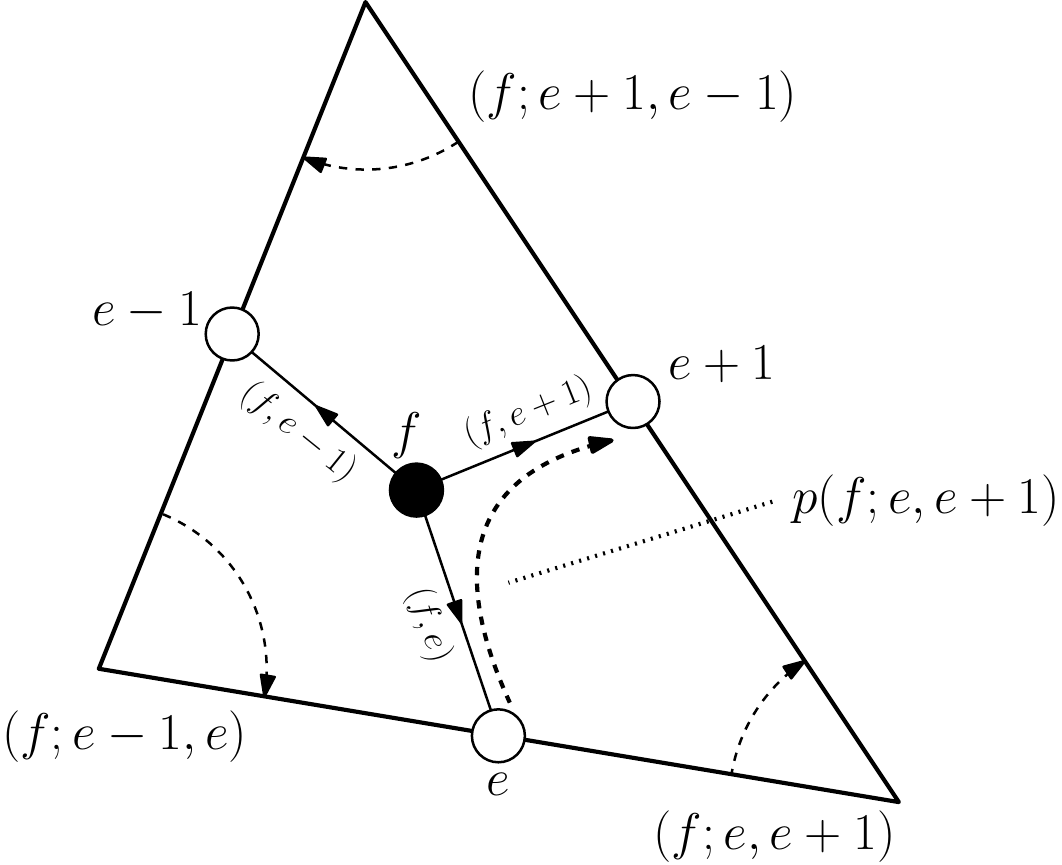}
    \caption{A face $f$ of a topological triangulation $\tau$, along with the dual graph $\Gamma(\tau)$ and the combinatorial angles $A(\tau)$.}
    \label{fig:triangle_schematic}
\end{figure}

\begin{defn}\label{defn:dual_graph}
Given a topological triangulation $\tau$ of $(S_g, Z)$, the \emph{dual graph} $\Gamma(\tau)$ is the trivalent ribbon graph with vertices $E(\tau) \cup F(\tau)$ and a half-edge $(f, e) \in H(\tau)$ from $f \in F$ to $e \in E$ if and only if $e$ is an edge of $f$.
The action of $\Z/3\Z$ on each $H_f$ is defined so that $(f, e + 1)$ is the edge of $f$ which is the successor of $(f, e)$ in the counterclockwise cyclic ordering on the edges of $f$.
We also let $A(\tau)$ denote the set of combinatorial angles of $\Gamma(\tau)$.
See Figure \ref{fig:triangle_schematic}.

For any trivalent ribbon graph $\Gamma$, it is easy to construct a unique topological triangulation $\tau(\Gamma)$ of some closed surface such that $\Gamma = \Gamma(\tau(\Gamma))$.
\end{defn}

\begin{rem}\label{rem:deformation_retraction}
As is implicit in Figure \ref{fig:triangle_schematic}, the graph $\Gamma(\tau)$ naturally embeds into $S_g \setminus Z$, and there is a deformation retraction $S_g \setminus Z \twoheadrightarrow \Gamma(\tau)$.
\end{rem}

The group of simplicial $0$-chains $C_0(\Gamma)$ is the free abelian group $\Z\la E \cup F \ra = \Z \la E \ra \oplus \Z \la F \ra$, and the group of simplicial $1$-chains $C_1(\Gamma)$ is the free abelian group $\Z \la H \ra$.
The following definition establishes our convention for the orientation of half-edges; in particular, we think of the half-edge $(f, e)$ as being directed from $f$ to $e$ (see $(f_1, e_1)$ and $(f_2, e_4)$ in Figure~\ref{fig:square_l_combinatorial}.)

\begin{defn}
We define the boundary map $\dee: C_1(\Gamma) \to C_0(\Gamma)$ by $\dee(f, e) = e - f$. Let us write $(e, f) \coloneqq -(f, e) \in C_1(\Gamma)$. We write $Z_1(\Gamma) = \ker \dee$, the group of simplicial 1-cycles.
\end{defn}

\begin{rem}
Since $\Gamma$ is 1-dimensional, no 1-cycle is a nontrivial boundary, and hence $H_1(\Gamma; \Z) = Z_1(\Gamma) \leq C_1(\Gamma)$.
\end{rem}

\begin{defn}
A \emph{simple cycle} in $\Gamma$ is an element $\alpha \in C_1(\Gamma)$ of the form
\[ \alpha = \sum_{j = 1}^\ell \left((e_{j-1}, f_j) + (f_j, e_j)\right), \]
where all $e_j$ and $f_j$ are distinct and the subscripts are considered modulo $\ell$.
\end{defn}

Note that no simple cycle is a nonunit multiple of any other, and a simple cycle may be expressible as a sum of other simple cycles. The following lemma is standard.

\begin{lem}
The subspace $H_1(\Gamma) \leq C_1(\Gamma)$ is spanned by the simple cycles in $\Gamma$.
\end{lem}

\subsection{Combinatorial holonomy}\label{subsec:combinatorial_holonomy}
In this subsection, we develop notation to express the holonomy formulas of \cite{Rivin} in the context of ribbon graphs.

\begin{defn}
An \emph{angle assignment} is a function $\theta: A \to (0, \pi)$ such that 
\[
\sum_{n \in \Z/3\Z} \theta(f; e + n, e + n + 1) = \pi.
\]
The \emph{angle space} $\mathcal A$ for $\Gamma$ is the set of all angle assignments $\theta: A \to (0,\pi)$.
\end{defn}

Note that an angle assignment $\theta$ extends naturally to a homomorphism $\Z\la A \ra \to \R$. In particular, we may write $\theta(f; e + 1, e) = -\theta(f; e, e + 1)$.
The following definition is inspired by \cite[Definition 6.6]{Rivin}.

\begin{defn}\label{contribution defn}
Given $\theta \in \mathcal A$, we define the following homomorphisms. The \emph{rotational contribution to holonomy} is the homomorphism
\begin{align*}
R^\theta: \Z\la A \ra &\to S^1 \subset \C^* \\
(f; e, e + 1) &\mapsto \exp(i\cdot\theta(f; e, e + 1)).
\end{align*}
The \emph{dilational contribution to holonomy} is the homomorphism
\begin{align*}
D^\theta: \Z\la A \ra &\to \R_{>0} \subset \C^* \\
(f; e, e + 1) &\mapsto \frac{\sin \theta(f; e + 1, e - 1)}{\sin \theta(f; e - 1, e)}.
\end{align*}
\end{defn}
These homomorphisms are defined to make Lemma~\ref{lem:computing_derivative} true. The following two definitions give formal relationships between homology cycles, combinatorial angles, and simplicial $1$-chains.
To illustrate this relationship, see the path labeled $p(f; e, e+1)$ in Figure~\ref{fig:schematic}.

\begin{defn}
    As $H_1(\Gamma)$ is generated by simple cycles, we define the homomorphism
    \[ \phi: H_1(\Gamma) \to \Z\langle A \rangle \]
    by extending the following action on simple cycles by linearity:
    \[ \phi\left(\sum_{j = 1}^\ell \left((e_{j-1}, f_j) + (f_j, e_j)\right)\right) \coloneqq \sum_{j = 1}^\ell (f_j; e_{j - 1}, e_j). \]
\end{defn}

\begin{defn}
    We define a homormorphism \begin{align*}
    p: \Z\la A \ra &\to C_1(\Gamma) \\
    (f; e, e+1) &\mapsto (f, e + 1) - (f, e)
    \end{align*}
    that sends an angle to its incident half-edges.
\end{defn}
The next lemma tells us that every cycle in $\Gamma$ can be understood as a sum of combinatorial angles.
\begin{lem}
    The homomorphism $\phi$ is injective.
\end{lem}
\begin{proof}
    If $\alpha = \sum_{j = 1}^\ell \left((e_{j-1}, f_j) + (f_j, e_j)\right)$
    is a simple cycle in $H_1(\Gamma)$, observe that
    \[ \alpha = \sum_{j = 1}^\ell \left((f_j, e_j) - (f_j, e_{j - 1})\right) = \sum_{j = 1}^\ell p(f_j; e_{j - 1}, e_j) = p\left(\phi(\alpha)\right), \]
    where arithmetic in the subscript $j$ is understood modulo $\ell$.
    In other words, $p$ is left inverse to $\phi$, and $\phi$ is injective.
\end{proof}

\begin{defn}\label{combinatorial holonomy}
We define the \emph{combinatorial rotational, dilational, and total holonomy} homomorphisms $H_1(\Gamma) \to \C^*$, respectively, as follows:
\[ \mathrm{Hol}^\theta_D = D^\theta \circ \phi,\quad \mathrm{Hol}^\theta_R = R^\theta \circ \phi,\quad \mathrm{Hol}^\theta = \mathrm{Hol}^\theta_D \cdot \mathrm{Hol}^\theta_R. \]
\end{defn}

The homomorphisms of this subsection fit into a commutative diagram:
\begin{center}
\begin{tikzcd}
H_1(\Gamma) \arrow[r, "\phi"', hook] \arrow[rd, hook] \arrow[rr, "\Hol^\theta", bend left] & \Z\la A \ra \arrow[d, "p"] \arrow[r, "D^\theta \cdot R^\theta"'] & \C^* \\
& C_1(\Gamma)  &   
\end{tikzcd}
\end{center}

\section{Complex-affine structures}\label{sec:complex_affine}
In this section, we develop geometric analogues to the combinatorial notions developed in \S\ref{sec:ribbon_graphs}.
In Definition~\ref{defn:induced_c_aff_struct}, we define a complex-affine structure on $S_g \setminus Z$ determined by an angle assignment $\theta: A \to (0, \pi)$.
In \S\ref{subsec:affine_holonomy}, we describe the relationship between its complex-affine holonomy and combinatorial holonomy.
In \S\ref{subsec:angle_map}, we establish precise notation to express the well-known correspondence between translation surfaces and complex-affine structures with trivial holonomy.

Let $\tau$ be a triangulation of $(S_g, Z)$, and let $\Gamma = \Gamma(\tau)$ be its dual graph.
Given $\theta \in \mathcal A$, we will define a complex-affine structure on $S_g \setminus Z$ with an associated flat connection $\nabla(\theta)$ on the tangent bundle $T(S_g \setminus Z)$.
\begin{defn}\label{defn:induced_c_aff_struct}
For each $(f,e) \in H$, let
\[ \phi_{(f,e)}^\theta: f \to \{z \in \C \st \Im(z) \geq 0\} \]
be the continuous orientation-preserving map whose image is the Euclidean triangle with angles $\{\theta(f; e + n, e + n + 1) \st 0 \le n \le 2\}$ such that $\phi_{(f,e)}(e)$ is the line segment $[0,1]$.

Now consider the combinatorial angle $(f; e, e') \in \Z\la A \ra$, and suppose that the edge $e'$ belongs to the faces $f$ and $f'$.
There is a unique complex-affine transformation
\[ \lambda_{(f; e, e')}^\theta: \C \to \C \]
that takes the edge $\phi_{(f, e)}^\theta(e')$ onto the edge $\phi_{(f', e')}^\theta(e')$ in an orientation-reversing manner.
(See Figure~\ref{fig:change_of_coordinates}.)

Gluing the family $F$ of triangles via the transformations $\lambda_{(f; e, e')}^\theta$ induces a complex-affine structure on the punctured surface $S_g \setminus Z$ as follows.
Consider small open neighborhoods $U_f$ of the punctured triangles $f \setminus Z \subset S_g \setminus Z$.
The maps $\phi_{(f, e)}^\theta$ extend to the sets $U_f$ in such a way that the transition functions $\phi^\theta_{(f', e')} \circ (\phi^\theta_{(f, e)})^{-1}$ are equal to the transformations $\lambda_{(f; e, e')}^\theta$.
We have thus specified an atlas of charts $\{\phi^\theta_{(f,e)}: U_f \to \C\}$ for a complex-affine structure $S_g \setminus Z$.

Finally, we define $\nabla(\theta)$ as the flat connection on $T(S_g \setminus Z)$ induced by this complex-affine structure.
\end{defn}

\begin{figure}
    \centering
    \includegraphics[width=0.60\linewidth]{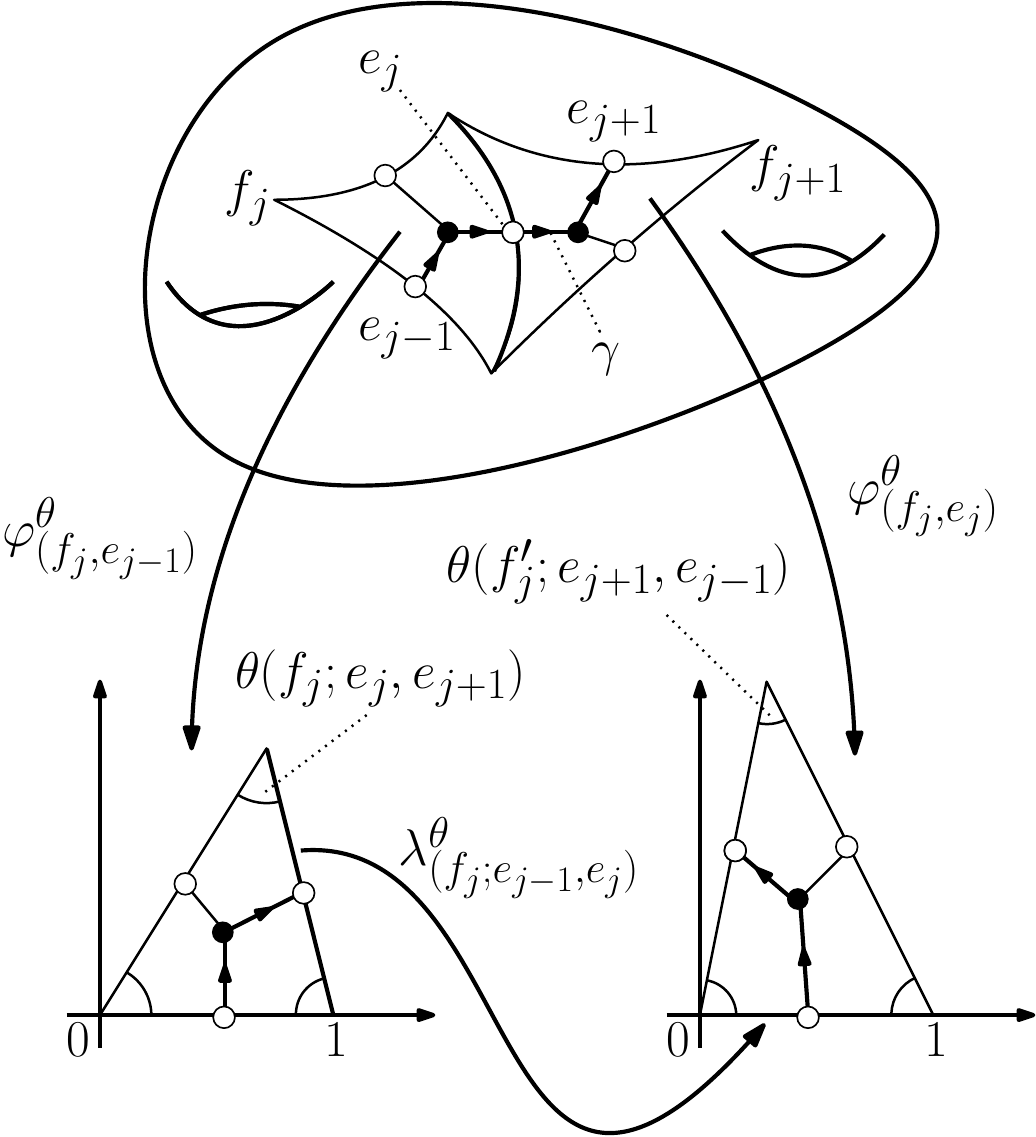}
    \caption{The complex-affine transition function determined by $\theta$ between two charts that meet the curve $\gamma$ in $\Gamma$.}
    \label{fig:change_of_coordinates}
\end{figure}
\begin{rem}
Note that, for edges $e$ and $e'$ of a common triangle $f$, the transition functions $\phi^\theta_{(f, e')} \circ (\phi^\theta_{(f, e)})^{-1}$ are complex-affine because the images of $f$ in each chart are similar triangles.
\end{rem}

\subsection{The holonomy homomorphism}\label{subsec:affine_holonomy}

The following computation is required for the proof of Lemma~\ref{lem:holonomy_is_holonomy}.
\begin{lem}\label{lem:computing_derivative}
    For all $(f; e, e') \in \Z\la A \ra$, we have
    \[ \frac{d}{dz}\lambda^\theta_{(f; e, e')} = D^\theta(f ; e, e') \cdot R^\theta(f ; e, e'). \]
\end{lem}
\begin{proof}
For $\zeta \in \C$, we refer to $\zeta/|\zeta|$ as the \emph{phase} of $\zeta$.
We will show that the phase and modulus of $\frac{d}{dz}\lambda^\theta_{(f; e, e')}$ are equal to $R^\theta(f; e, e')$ and $D^\theta(f; e, e')$, respectively.

We first observe that $\lambda_{(f; e, e+1)}^\theta$ rotates the triangle $\phi_{(f, e)}^\theta(f)$ counterclockwise through the angle $\theta(f; e, e+1)$ and $\lambda_{(f; e, e-1)}^\theta$ rotates $\phi_{(f,e)}^\theta(f)$ clockwise through the angle $\theta(f; e, e-1)$.
Since $e'$ is either $e+1$ or $e-1$, it follows from either case that the phase of $\frac{d}{dz}\lambda^\theta_{(f; e, e')}$ is equal to $R^\theta(f; e, e')$.

Now observe that since the segment $\phi_{(f, e)}^\theta(e)$ has length $1$, the Law of Sines implies that the length of $\phi_{(f, e)}^\theta(e + 1)$ is $\sin\theta(f; e - 1, e)/\sin\theta(f; e+1, e - 1)$, and the length of $\phi_{(f, e)}^\theta(e - 1)$ is equal to $\sin\theta(f; e, e+1)/\sin\theta(f; e+1, e - 1)$. 
In particular, we have
\[
\mathrm{length}\left(\phi_{(f,e)}^\theta(e')\right) = D^\theta(f; e, e')\1.
\]
Since the map $\lambda_{(f; e, e')}^\theta$ takes the edge $\phi_{(f,e)}^\theta(e')$ onto the segment $[0,1]$, it follows that the modulus of $\frac{d}{dz}\lambda^\theta_{(f; e, e')}$ is equal to the inverse of the length of $\phi^\theta_{(f, e)}(e')$, which is $D^\theta(f; e, e')$ as desired.
\end{proof}

\begin{defn}\label{nabla definition}
Fixing an arbitrary basepoint $x \in S_g \setminus Z$, we have a holonomy homomorphism $\pi_1(S_g \setminus Z, x) \to \C^*$ given by parallel transport along the flat connection $\nabla(\theta)$.
Since $\C^*$ is abelian, this map factors through the abelianization $H_1(S_g \setminus Z; \Z)$ of $\pi_1(S_g \setminus Z, x)$.
We define $\mathrm{Hol}^{\nabla(\theta)}: H_1(S_g \setminus Z; \Z) \to \C^*$ to be the induced map.
\end{defn}

As in Remark~\ref{rem:deformation_retraction}, there is a deformation retraction $S_g \setminus Z \twoheadrightarrow \Gamma$, and hence the inclusion $i: \Gamma \hookrightarrow S_g \setminus Z$ induces an isomorphism $i_*: H_1(\Gamma) \xrightarrow{\sim} H_1(S_g \setminus Z; \Z)$.
The following lemma asserts that the combinatorial holonomy map $\mathrm{Hol}^\theta$ does indeed compute the holonomy of $\nabla(\theta)$.

\begin{lem}\label{lem:holonomy_is_holonomy}
We have
\[ \mathrm{Hol}^{\nabla(\theta)} \circ i_* = \mathrm{Hol}^\theta. \]
\end{lem}

\begin{proof}
Let $\alpha = \sum_{j=1}^\ell \left((e_{j-1}, f_j) + (f_j, e_j)\right) \in H_1(\Gamma)$ be a simple cycle.
Then $i_*(\alpha)$ is the homology class $[\gamma] \in H_1(S_g \setminus Z; \Z)$ of the simple closed curve $\gamma$, which is the concatenation of arcs that join the edges $e_{j-1}$ and $e_j$ of $f_j$ (as in Figure~\ref{fig:change_of_coordinates}).

Note that $\gamma$ passes through the open sets $\{U_{f_j} \st j \in \Z/\ell\Z\}$, and the transition function from $U_{f_j}$ to $U_{f_{j+1}}$ is $\lambda_{(f_j; e_{j-1}, e_j)}^\theta$.
By definition, $\mathrm{Hol}^{\nabla(\theta)}([\gamma])$ is equal to the derivative of the composition of the transition functions between the charts through which $\gamma$ passes.
Therefore,
\begin{align*}
\mathrm{Hol}^{\nabla(\theta)}(i_*(\alpha)) = \mathrm{Hol}^{\nabla(\theta)}([\gamma]) &= \frac{d}{dz} \left(\lambda_{(f_\ell; e_{\ell - 1}, e_\ell)}^\theta \circ \cdots \circ \lambda_{(f_1; e_0, e_1)}^\theta \right)\\
 &= \prod_{j=1}^\ell D^\theta(f_j; e_{j-1}, e_j) \cdot R^\theta(f_j; e_{j-1}, e_j) \\
&= \mathrm{Hol}_D^\theta(\alpha) \cdot \mathrm{Hol}_R^\theta(\alpha) \\
&= \mathrm{Hol}^\theta(\alpha),
\end{align*}
where the third equality follows from Lemma~\ref{lem:computing_derivative}.
\end{proof}

\subsection{The angle map}\label{subsec:angle_map}
Fix a topological triangulation $\tau$, and let $\Delta = \Delta(\tau)$ be its associated iso-triangulable region.
We now consider how $\Delta$ embeds into the space of complex-affine structures under the \emph{angle map} $\Theta$.
Lemma~\ref{lem:image_of_angle_map} uses $\Theta$ to express the well-understood fact that a translation surface structure is equivalent (up to the action of $\C^*$) to a complex-affine structure with trivial holonomy.

\begin{defn}\label{defn:angle_map}
Let $\Tilde\Delta = \Tilde\Delta(\tau) \subset \Omega \mathcal T_{g,n}(\kappa)$ be an iso-triangulable region.
For each Teichm\"uller equivalence class $M \in \Tilde\Delta$, choose a translation surface structure $(X,\omega) \in M$ such that $\tau$ is a geometric triangulation of $(X,\omega)$.
For each combinatorial angle $(f; e, e+1) \in A$, we define the map
\[ \Tilde\Theta_{(f; e, e+1)}: \Tilde\Delta \to (0, \pi) \]
such that $\Tilde\Theta_{(f; e, e+1)}(M)$ is equal to the angle on $(X,\omega)$ between the line segments $e, e + 1$, measured inside the Euclidean triangle $f$. Note that the angle is independent of the choice of $(X, \omega) \in M$. These $\Tilde\Theta_{(f; e, e+1)}$ together define the \emph{angle map} $\Tilde\Theta: \Tilde\Delta \to \mathcal A$, where $\Tilde\Theta(M)$ is the angle assignment $(f; e, e + 1) \mapsto \Tilde\Theta_{(f; e, e+1)}(M)$.

The angle map $\Tilde \Theta$ is invariant under the action of $\C^*$ on $\Omega \mathcal T_{g,n}(\kappa)$, because this action simply rescales and rotates the triangles, and so we have an induced angle map
\[ \Theta: \Delta \to \mathcal A \]
defined on $\Delta \subset \P \Omega \mathcal T_{g,n}(\kappa)$.
\end{defn}

As a consequence of Lemma~\ref{lem:image_of_angle_map}, we may understand $\Theta$ as providing a system of coordinates on $\Delta$, which we refer to as \emph{angle coordinates}.

\begin{defn}
We define the \emph{locus of trivial holonomy} in $\mathcal A$ to be the set
\[ \mathcal V = \left\{ \theta \in \mathcal A \mst \mathrm{Hol}^{\theta}(\alpha) = 1 \text{ for all } \alpha \in H_1(\Gamma) \right\}. \]
That is, $\mathcal V$ is the preimage under $\mathrm{Hol}: \mathcal A \to \Hom(H_1(\Gamma), \C^*)$ of the trivial character $H_1(\Gamma) \to \C^*$.
\end{defn}

\begin{lem}\label{lem:image_of_angle_map}
The map $\Theta: \Delta \to \mathcal A$ is a homeomorphism onto the locus $\mathcal V$ of trivial holonomy.
\end{lem}
\begin{proof}
We first prove that $\Theta$ is continuous.
Given $(f; e, e+1)$, let $\gamma$ and $\gamma'$ denote the relative homology classes of the edges $e$ and  $e+1$ of  $f$, respectively, oriented outward from their common vertex.
Writing $\Phi_\gamma(M) = r\exp(i\phi)$ and $\Phi_{\gamma'}(M) = r'\exp(i\phi')$, we have $\Tilde\Theta_{(f; e, e+1)}(M) = \phi - \phi' \in (0, \pi)$, where this difference is a well-defined element of $(0, \pi)$ because $e+1$ follows $e$ in the counterclockwise orientation of the face $f$.
It follows that $\Tilde \Theta$ depends continuously on the period maps $\Phi_\gamma$, and so $\Tilde \Theta$ is continuous by the definition of the topology on $\Omega \mathcal T_{g,n}(\kappa)$.
Therefore, the induced angle map $\Theta$ is also continuous, as desired.

Next, we show that the image of $\Theta$ is in $\mathcal{V}$.
Consider $M \in \Delta$, and let $(X,\omega)$ be a translation surface structure in the Teichm\"uller equivalence class $M$ such that $\tau$ is a geometric triangulation of $(X,\omega)$.
Recall that the transition functions for the atlas that induce the translation structure on $(X, \omega)$ are of the form $z \mapsto z + C$ for $C \in \C$.
As these maps have derivative equal to $1 \in \C^*$, Lemma~\ref{lem:holonomy_is_holonomy} implies that $\Theta(M) \in \mathcal V$.

Finally, we construct a continuous inverse mapping to $\Theta$ that is defined on $\mathcal{V}$.
Fix $\theta \in \mathcal V$, and consider the complex-affine structure from Definition~\ref{defn:induced_c_aff_struct}, which induces the structure of a Riemann surface $X$ on $S_g$.
Choose arbitrarily a half-edge $(f, e) \in H$, and consider the differential form on the face $f \subset S_g$ given by $\omega_{(f, e)} \coloneqq (\phi_{(f, e)}^\theta)^*(dz)$, where the edge $e$ of $f$ is arbitrarily chosen.
The holonomy of the flat connection $\nabla(\theta)$ is precisely the cohomological obstruction to extending this local form to a well-defined global form $\omega$ on all of $S_g$.
Since $\Hol^\theta = \Hol^{\nabla(\theta)} \equiv 1$, there is no obstruction, and therefore we obtain a globally well-defined differential form $\omega$ on $S_g$.
If another half-edge $(f, e)$ had been chosen, the resulting form would differ by a scalar $\lambda \in \mathbb{C}^*$.
The mapping $\mathcal V \to \Delta$ that sends $\theta$ to $[(X, \omega)]$, the Teichm\"uller equivalence class of $(X,\omega)$, is inverse to $\Theta$ by construction, and is clearly continuous.
We conclude that $\Theta$ is a homeomorphism onto $\mathcal V$.
\end{proof}

\section{Triangle matchings}\label{sec:triangle_matchings}
We are now ready to define a \emph{triangle matching} $\iota : H(\tau) \to H(\tau)$, the central tool of this paper.
In \S\ref{subsec:matchings_angles}, we define $\iota$-invariant angle assignments, and we prove in Theorem~\ref{thm:constant_holonomy_on_invariants} that they determine the same combinatorial holonomy.
In Lemma~\ref{lem:Sam's_favorite_formula} of \S\ref{subsec:matchings_rel_homology}, we show that a triangle matching induces a symmetry on the relative homology classes of the edges of $\tau$, which will be used to establish Lemma~\ref{lem:image_of theta_inside_invariant} and Lemma~\ref{lem:same_cyls}.

\begin{defn}\label{defn:triangle_matching}
Let $\iota: H \to H$ be a bijection. We say that $\iota$ is a \emph{triangle matching} if
\begin{enumerate}
\item\label{commutativity} $\iota$ commutes with the action $\Z/3\Z \curvearrowright H$.
\item\label{restriction as -1} The induced automorphism $\iota: \Z\la H \ra = C_1(\Gamma) \to C_1(\Gamma)$ satisfies $\iota(\alpha) = -\alpha$ for all $\alpha \in H_1(\Gamma) \leq C_1(\Gamma)$.
\end{enumerate}
\end{defn}

\begin{rem}
The name ``triangle matching'' was chosen for the following reasons. Since, by Property~(\ref{commutativity}), $\iota$ commutes with the action $\Z/3\Z \curvearrowright H$, and the set of triangles $F$ is in one-to-one correspondence with the $\Z/3\Z$-orbits of half-edges, the map $\iota$ induces a bijection $\iota_F: F \to F$.

In fact, the triangles $f$ and $\iota_F(f)$ are ``matched'' in a stronger geometric sense: they must be congruent in any translation surface structure $(X, \omega)$ for which $\tau$ is a geometric triangulation. This is an immediate consequence of Lemma~\ref{lem:Sam's_favorite_formula}, because the shapes of the triangles are determined by the periods $\int_{r(f, e)} \omega = -\int_{r(\iota(f, e))} \omega$.
\end{rem}


\begin{figure}
    \centering
    \includegraphics[width=0.5\linewidth]{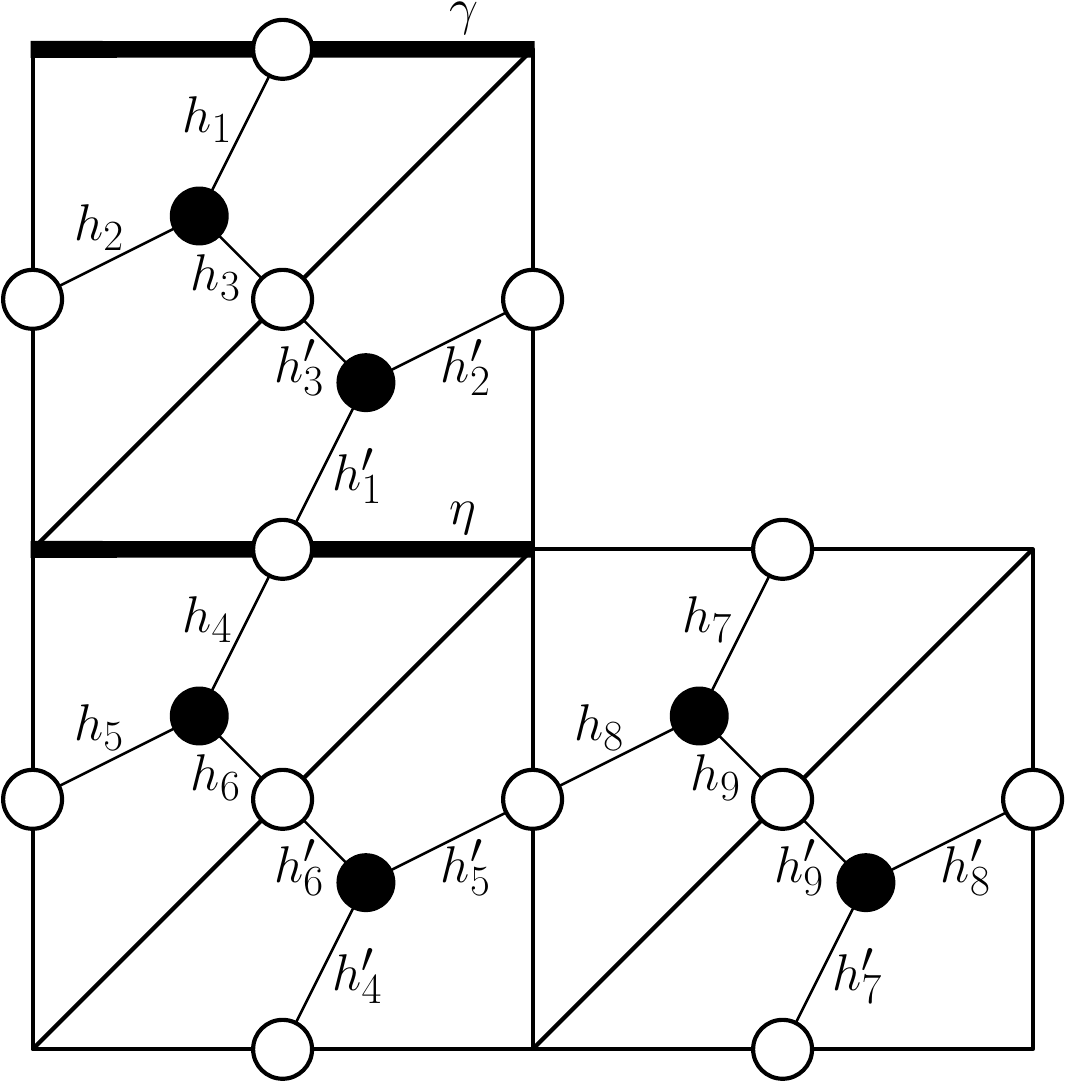}
    \caption{A triangle matching $\iota$ on the ribbon graph dual to a triangulation of a L-shaped translation surface.
    In the figure, we write $\iota(h_j) = h'_j$.
    Because the half-edges $h_1$ and $h_1'$ are matched, Lemma~\ref{lem:Sam's_favorite_formula} implies that the arcs $\gamma$ and $\eta$ (when oriented left to right) are homologous.}
    \label{fig:square_l}
\end{figure}

\begin{ex}(Square L)\label{ex:square_l_matching}
In Figure~\ref{fig:square_l}, we indicate a triangle matching for the dual graph of the standard triangulation of the square L.
From the figure, we can see that the cyclic order about each solid face vertex is preserved.
The reader can also check that $\iota$ negates the homology classes of each horizontal and vertical cylinder core curve (which form a basis for $H_1(\Gamma)$ in this case).
For example, the horizontal core curve for the top square is the simple cycle
\[ \alpha \coloneqq -h_2 + h_3 - h_3' + h_2'. \]
Applying $\iota$ yields
\[ \iota(\alpha) = -h'_2 + h'_3 - h_3 + h_2, \]
so $\iota(\alpha)$ is indeed equal to $-\alpha$.
\end{ex}

\subsection{Triangle matchings and holonomy}\label{subsec:matchings_angles}
Every triangle matching naturally determines an automorphism of the angle group $\Z \la A \ra$ as in Definition~\ref{defn:angle_matching}.
We may therefore consider the set of $\iota$-invariant angle assignments $\mathcal{A}^\iota$.
In Lemma~\ref{lem:convex_invariant_set}, we show that $\mathcal A^\iota$ is a convex subspace of the space of angle assignments $\mathcal{A}$.
We study the representation $\theta \mapsto \mathrm{Hol}^\theta \in \mathrm{Hom}(H_1(\Gamma), \C^*)$ on $\mathcal{A}^\iota$ in Theorem~\ref{thm:constant_holonomy_on_invariants} to establish Corollary~\ref{cor:invariant_inside_trivial}, which is one direction of the containment to determine the image $\Theta(\Delta)$.

\begin{defn}\label{defn:angle_matching}
Let $(f', e') = \iota(f, e)$.
Because $\iota$ commutes with the $(\Z/3\Z)$-action on $H$, we have $\{\iota(f, e), \iota(f, e) + 1\} = \{(f', e'), (f', e' + 1)\}$.
That is to say, $\iota$ induces a well-defined bijection $\iota: A \to A$ given by $\iota(f; e, e + 1) = (f'; e', e' + 1)$, and hence an automorphism $\iota: \Z\la A \ra \to \Z\la A \ra$.
\end{defn}

\begin{defn}
We say that an angle assignment $\theta: A \to (0, \pi)$ is $\iota$\emph{-invariant} if $\theta \circ \iota = \theta$. We denote by $\mathcal A^\iota$ the set of all $\iota$-invariant angle assignments $\theta \in \mathcal A$.
\end{defn}

Note that $\mathcal A$ can be viewed as the affine-linear subspace of $(0,\pi)^{3n} \subset \R^{3n}$, where $n = \# F$, consisting of points $x$ that satisfy $x_{3j} + x_{3j + 1} + x_{3j + 2} = \pi$ for all $0 \leq j < n$.
Therefore, we may speak of convex subsets of $\mathcal A$.

\begin{lem}\label{lem:convex_invariant_set}
Let $\iota: H \to H$ be a triangle matching. Then $\mathcal A^\iota$ is a convex subspace of $\mathcal A$.
\end{lem}

\begin{proof}
The set $\mathcal A^\iota$ is defined by a system of homogeneous linear equations, namely
\[
\theta(f; e, e+1) - \theta(\iota(f; e, e+1)) = 0
\]
for all $(f; e, e+1) \in A$.
\end{proof}

The holonomy representation restricted to $\mathcal A^\iota$ takes a particularly simple form:

\begin{lem}\label{lem:order_2_holonomy}
Let $\theta \in \mathcal A^\iota$.
Then $\mathrm{Hol}^\theta_D: H_1(\Gamma) \to \mathbb{R}_{> 0}$ is trivial, and $\mathrm{Hol}^\theta_R(\alpha) = \pm 1$ for all $\alpha \in H_1(\Gamma)$.
\end{lem}

\begin{proof}
Since $\theta \circ \iota = \theta$, we also have $\mathrm{Hol^\theta} \circ \iota = \mathrm{Hol}^{\theta \circ \iota} = \mathrm{Hol}^\theta$.
So for $\alpha \in H_1(\Gamma)$, we have
\[
\mathrm{Hol}^\theta(\alpha) = \mathrm{Hol}^\theta(\iota(\alpha)) = \mathrm{Hol}^\theta(-\alpha) = \mathrm{Hol}^\theta(\alpha)^{-1}.
\]
The multiplicative group $\R_{> 0}$ does not have elements of order 2 other than the identity, so $\mathrm{Hol}^\theta_D$ is trivial.
The elements of order 2 in $S^1$ are $\pm 1$, and hence the latter claim follows.
\end{proof}

Although rotational holonomy has an a priori dependence on the angles, it is surprisingly independent of $\theta \in \mathcal{A}\iota$:

\begin{thm}\label{thm:constant_holonomy_on_invariants}
Let $\theta, \theta' \in \mathcal A^\iota$. Then $\mathrm{Hol}^\theta = \mathrm{Hol}^{\theta'}$.
\end{thm}

\begin{proof}
Consider the representation $\mathrm{Hol}: \mathcal A \to \Hom(H_1(\Gamma), \C^*) \cong (\C^*)^{\mathrm{rk} H_1(\Gamma)}$ sending $\theta$ to $\mathrm{Hol}^\theta$. From the formulas in Definition~\ref{contribution defn}, we see that $\mathrm{Hol}$ is continuous. From Lemma~\ref{lem:order_2_holonomy} it follows that the image of $\mathcal A$ under $\mathrm{Hol}$ is discrete. Finally, Lemma~\ref{lem:convex_invariant_set} implies that $\mathcal A^\iota$ is connected, and so we conclude that $\mathrm{Hol}|_{\mathcal A^\iota}$ is constant.
\end{proof}

An $\iota$-invariant angle assignment need not have trivial holonomy, but as an immediate corollary of Theorem~\ref{thm:constant_holonomy_on_invariants} we have the following.
\begin{cor}\label{cor:invariant_inside_trivial}
    If $\mathcal \mathcal A^\iota \cap \mathcal{V}$ is nonempty, then $\mathcal A^\iota \subseteq \mathcal V$.
\end{cor}

\subsection{Triangle matchings and relative homology}\label{subsec:matchings_rel_homology}
Every half-edge $(f, e)$ determines a unique class $r(f, e) \in H_1(S_g, Z; \Z)$ that we define in Definition~\ref{defn:relative_homology_map}.
In Lemma~\ref{lem:Sam's_favorite_formula}, we demonstrate the identity $r(f,e) = -r(\iota(f,e))$, which we will use in Lemma~\ref{lem:image_of theta_inside_invariant} to prove that the image $\Theta(\Delta(\tau))$ of an iso-triangulable region is $\iota$-invariant whenever $\tau$ admits a triangle matching.

\begin{defn}\label{defn:relative_homology_map}
Given $(f, e) \in C_1(\Gamma)$, let $r(f, e) \in H_1(S_g, Z; \Z)$ be the relative homology class of the edge $e$ of $\tau$, oriented according to the counterclockwise orientation of the face $f$.
\end{defn}
Observe that if $e$ belongs to the faces $f$ and $f'$ of $\tau$, then $r(f, e) = -r(f', e)$.

\begin{lem}\label{lem:Sam's_favorite_formula}
Let $\iota: H \to H$ be a triangle matching. Then $r(f,e) = -r(\iota(f,e))$ for every $(f,e) \in H$.
\end{lem}

\begin{proof}
Recall that there is a natural isomorphism $i_* : H_1(\Gamma) \xrightarrow{\sim} H_1(S_g \setminus Z; \Z)$.
Therefore, the algebraic intersection pairing on $H_1(S_g \setminus Z; \Z) \times H_1(S_g, Z; \Z)$ gives a pairing
\[ \la \cdot, \cdot \ra: H_1(\Gamma) \times H_1(S_g, Z; \Z) \to \Z, \]
which is nondenerate by Poincar\'e duality.

Let $\alpha =  \sum \left((e_{j-1}, f_j) + (f_j, e_j)\right)$ be a simple cycle.
Letting $[f, e](\alpha)$ denote the coefficient of $(f,e)$ in this sum, observe that
\[ \la \alpha, r(f,e) \ra = [f, e](\alpha). \]
Since $\iota(\alpha) = -\alpha$, we have that
\[ [f, e](\alpha) = [\iota(f, e)](\iota(\alpha)) = [\iota(f, e)](-\alpha) = -[\iota(f, e)](\alpha). \]
We have thus showed that
\[ \la \alpha, r(f, e) \ra = -\la \alpha, r(\iota(f, e)) \ra \]
for every simple cycle $\alpha$.
Since $H_1(\Gamma)$ is spanned by simple cycles and the intersection pairing is nondegenerate, we conclude that $r(f, e) = -r(\iota(f, e))$.
\end{proof}

For example, the horizontal arcs $\gamma \coloneqq -r(h_1)$ and $\eta \coloneqq r(h_1')$ in Figure~\ref{fig:square_l} are homologous.

\section{Proof of Theorem~\texorpdfstring
    {\hyperref[thm:mainA]{\ref*{thm:mainA}}}
    {\ref*{thm:mainA}}}
    \label{sec:proof_of_main_theorem}
Fix an iso-triangulable region $\Delta = \Delta(\tau)$ with a triangle matching $\iota: H(\tau) \to H(\tau)$.
Lemma~\ref{lem:image_of_angle_map} and Corollary~\ref{cor:invariant_inside_trivial} together imply that $\Theta(\Delta) \cong \mathcal V \supseteq \mathcal A^\iota$.
In Lemma~\ref{lem:image_of theta_inside_invariant}, we establish the reverse inclusion $\mathcal V \subseteq \mathcal{A}^\iota$.
We conclude our proof of Theorem~\ref{thm:mainA} in \S\ref{subsec:iso_Delaunay_regions} by extending the result to iso-Delaunay regions in Theorem~\ref{thm:main_theorem_precise}.

\subsection{Iso-triangulable regions}\label{subsec:iso_triangulable_regions}

\begin{lem}\label{lem:image_of theta_inside_invariant}
If $\tau$ has a triangle matching $\iota : H(\tau) \to H(\tau)$, then $\mathcal V \subseteq \mathcal{A}^\iota$.
\end{lem}
\begin{proof}
Fix a combinatorial angle $(f; e, e+1)$, and write $(f', e') = \iota(f, e)$.
Recall that Lemma~\ref{lem:Sam's_favorite_formula} implies 
\[ r(f', e') = -r(f, e). \]
Since $\iota$ respects the orientation of the faces of $\tau$, we conclude that $\theta(f; e, e+1) = \theta(f'; e', e' + 1)$ and $\theta$ is $\iota$-invariant.
\end{proof}

\begin{cor}\label{cor:trivial_equals_invariant}
If $\tau$ has a triangle matching, then $\mathcal{V} = \mathcal{A}^\iota$.
\end{cor}


\subsection{Iso-Delaunay regions}\label{subsec:iso_Delaunay_regions}
Given $e \in E$ with $e \subset f, f'$ and $\theta \in \mathcal A$, we define the \emph{sum of} $\theta$\emph{-angles opposite} $e$ by
\[ \theta(e) \coloneqq \theta(f; e+1, e+2) + \theta(f'; e+1, e+2) \in \Z\la A \ra. \]
The following is a fact of Euclidean geometry.

\begin{prop}\label{Euclidean delaunay}
Given a triangle $\triangle ABC$ in the plane, let $D$ be a point such that $A$ and $D$ lie on opposite sides of the line $\overleftrightarrow{BC}$. Then $D$ lies outside the circumcircle of $\triangle ABC$ if and only if $\angle(\overline{AB}, \overline{AC}) + \angle(\overline{DC}, \overline{DB}) < \pi$.
\end{prop}

\begin{defn}
Let $\mathcal D \subset \mathcal A$ be the set of all $\theta \in \mathcal A$ such that $\theta(e) < \pi$ for all $e \in E$.
\end{defn}

The following is an immediate corollary to Proposition~\ref{Euclidean delaunay}.
See Figure~\ref{fig:delaunay}.

\begin{cor}\label{cor:delaunay_criterion}
Let $M = [(X, \omega)] \in \Delta$, and let $\theta = \Theta(M)$. Then $M \in \Delta_{\mathcal D}$ if and only if $\Theta(M) \in \mathcal D$.
\end{cor}



We are now ready to prove Theorem~\ref{thm:mainA}.

\begin{thm}[Theorem A]\label{thm:main_theorem_precise}
Let $\tau$ be a topological triangulation of $(S_g, Z)$, and suppose that there exists a triangle matching $\iota: H(\tau) \to H(\tau)$.
Then $\Theta|_{\Delta_{\mathcal D}(\tau)}$ is a homeomorphism from $\Delta_{\mathcal D}(\tau)$ onto a convex subspace of Euclidean space.
Consequently, $\Delta_{\mathcal D}(\tau)$ is a connected contractible space.
\end{thm}

\begin{proof}
Since $\Theta$ is a homeomorphism from $\Delta(\tau)$ onto $\mathcal V$, it follows from Corollary~\ref{cor:delaunay_criterion} that $\Theta|_{\Delta_{\mathcal D}(\tau)}$ is a homeomorphism from $\Delta_{\mathcal D}(\tau)$ onto $\mathcal D \cap \mathcal V$. By Corollary~\ref{cor:trivial_equals_invariant}, we have $\mathcal D \cap \mathcal V = \mathcal D \cap \mathcal A^\iota$. The set $\mathcal D$ is a convex subset of $\mathcal A$, and the set $\mathcal{A}^\iota$ is convex by Lemma~\ref{lem:convex_invariant_set}.
We conclude that $\Theta(\Delta_{\mathcal D}(\tau)) = \mathcal D \cap \mathcal A^\iota$ is convex.
\end{proof}

\section{Examples of convex iso-Delaunay regions}\label{sec:examples}
In this section we give infinitely many examples of triangulations that admit triangle matchings, and hence by Theorem~\ref{thm:main_theorem_precise}, examples of convex iso-Delaunay regions.
In \S\ref{subsec:hyperelliptic}, we show that Delaunay triangulations of translation surfaces in hyperelliptic stratum components admit triangle matchings.
In \S\ref{subsec:connected_sum}, we define the connected sum $\Gamma_1 \# \Gamma_2$ of trivalent ribbon graphs and show that if $\Gamma_1$ and $\Gamma_2$ admit triangle matchings, then so does $\Gamma_1 \# \Gamma_2$.
In \S\ref{subsec:arboreal}, we combine the results of the previous subsections to show that a certain family of \emph{origamis} (also known as \emph{square-tiled surfaces}) called \emph{arboreal} origamis admits triangle matchings.

\subsection{Hyperelliptic examples}\label{subsec:hyperelliptic}
We will show in Lemma~\ref{lem:hyperelliptic_top} that hyperelliptic involutions acting simplicially on a triangulation yield examples of triangle matchings, from which Corollary~\ref{cor:hyperelliptic} follows.
Note that this family of examples includes the square L.
We start by recalling some definitions.

\begin{defn}\label{Hyperelliptic involution}
An orientation-preserving homeomophism $\iota: S_g \to S_g$ is a \emph{hyperelliptic involution} if $\iota_*(\beta) = -\beta$ for every $\beta \in H_1(S_g; \Z)$.
\end{defn}

\begin{defn}\label{hyperelliptic defn}
Let $\tau$ be a topological triangulation of $(S_g, Z)$, and let $\iota: S_g \to S_g$ be a hyperelliptic involution.
We say that $\tau$ is \emph{compatible with} $\iota$ if
\begin{enumerate}
\item \label{simplicial automorphism} the map $\iota$ induces a simplicial automorphism $\tau \to \tau$, and
\item \label{geometric hyperellipticity} either $\# Z= 1$, or $\# Z = 2$ and $\iota$ nontrivially permutes the elements of $Z$.
\end{enumerate}
\end{defn}

\begin{lem}\label{lem:hyperelliptic_top}
Suppose that $\tau$ is a topological triangulation of $(S_g, Z)$ that is compatible with a hyperelliptic involution $\iota: S_g \to S_g$.
Then $\iota$ induces a graph automorphism $\iota:\Gamma(\tau) \to \Gamma(\tau)$ such that $\iota: H(\tau) \to H(\tau)$ is a triangle matching.
\end{lem}

\begin{proof}
The simplicial automorphism $\iota:\tau \to \tau$ induces a graph automorphism $\iota: \Gamma(\tau) \to \Gamma(\tau)$. Thus, there is a well-defined induced map $\iota: H \to H$ on the half-edges of $\Gamma$, and this map commutes with the $(\Z/3\Z)$-action on $H$ because $\iota: S_g \to S_g$ is orientation-preserving.

Suppose $\# Z = 1$. Then we have natural isomorphisms
\[ H_1(S_g; \Z) \cong H_1(S_g \setminus Z; \Z) \cong H_1(\Gamma; \Z). \]
By Definition~\ref{Hyperelliptic involution}, we have $\iota(\alpha) = -\alpha$ for every $\alpha \in H_1(\Gamma; \Z)$, and therefore $\iota: H \to H$ is a triangle matching.

Now suppose $Z = \{x_1, x_2\}$. Then there is a natural surjection
\[
H_1(S_g \setminus Z; \Z) \overset{\pi}{
\twoheadrightarrow} H_1(S_g; \Z),
\]
whose kernel is isomorphic to $\Z$. Let $\beta_j \in H_1(S_g \setminus Z; \Z)$ denote the homology class of a small arc $b_j \subset S_g \setminus Z$ encircling the point $z_j$ counterclockwise for $j = 1,2$.
Then $\beta_1 = -\beta_2$, because $b_1$ and $b_2$ co-bound a connected subsurface of $S_g \setminus Z$.
In addition, $\ker(\pi)$ is equal to the set of integer multiples of $\beta_1$, because $b_1$ bounds a disk in $S_g$ centered on $z_1$.
Since $\iota$ is assumed to nontrivially permute the elements of $Z$, we have $\iota(z_1) = z_2$ and also $\iota(\beta_1) = \beta_2$.
Since $H_1(S_g \setminus Z; \Z) \cong H_1(\Gamma; \Z)$, the fact that $\iota$ is a hyperelliptic involution along with the identity $\iota(\beta_1) = \beta_2 = -\beta_1$ implies that $\iota(\alpha) = -\alpha$ for every $\alpha \in H_1(\Gamma; \Z)$.
We conclude that $\iota: H \to H$ is a triangle matching.
\end{proof}

The following definition is equivalent to \cite[\S 2.1 Definition 2]{KontsevichZorich}.

\begin{defn}
A translation surface structure $(X, \omega)$ on $(S_g, Z)$ \emph{lies in a hyperelliptic component} if there exists a hyperelliptic involution $\iota: S_g \to S_g$ so that the following conditions hold.
\begin{enumerate}
    \item The map $\iota: X \to X$ is a biholomorphism satisfying $\iota^*\omega = -\omega$.
    \item Either $\# Z= 1$, or $\# Z = 2$ and $\iota$ nontrivially permutes the elements of $Z$.
\end{enumerate}
\end{defn}

\begin{lem}\label{lem:hyperelliptic_delaunay}
Let $(X, \omega)$ be a translation surface structure on $(S_g, Z)$ with nondegenerate Delaunay triangulation $\tau$. If $(X, \omega)$ lies in a hyperelliptic component with hyperelliptic involution $\iota: S_g \to S_g$, then $\tau$ is compatible with $\iota$.
\end{lem}
\begin{proof}
    Since $\iota^* \omega = -\omega$, the involution $\iota$ acts as an isometry for the flat metric on $(X, \omega)$.
    By Remark~\ref{rem:uniqueness}, the triangulation $\tau$ is determined by the metric structure of $(X, \omega)$, and hence is isometry-invariant.
    Thus $\iota$ acts on $\tau$ as a simplicial automorphism.
\end{proof}

The above lemma immediately implies Corollary~\ref{cor:hyperelliptic}.

\subsection{Connected sums of ribbon graphs}\label{subsec:connected_sum}
We now describe a combinatorial operation on ribbon graphs called  \emph{connected sum} in order to produce new triangle matchings from old ones.
Our proof in Lemma~\ref{lem:connected_sum_triangle_matching} uses a standard Mayer--Vietoris argument; Figure~\ref{fig:schematic} illustrates the idea of the proof.
\begin{figure}
    \centering
    \includegraphics[width=0.45\linewidth]{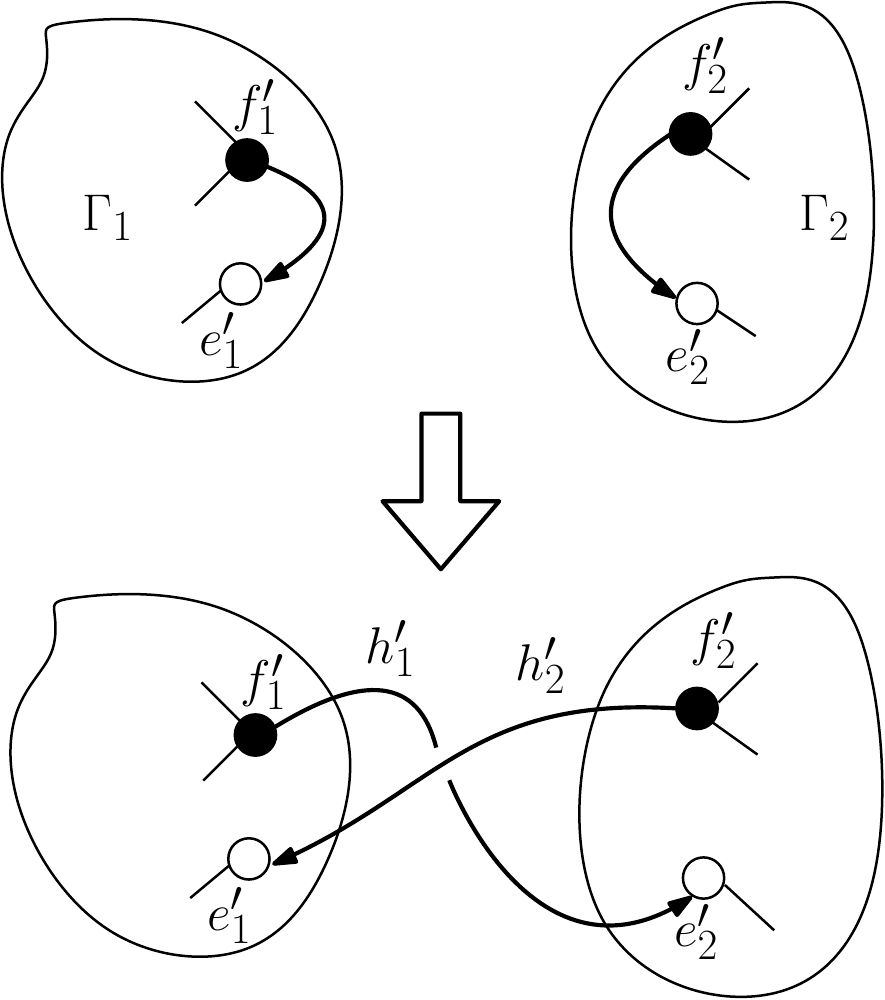}
    \caption{A schematic of the sum of two ribbon graphs $\Gamma_1$ and $\Gamma_2$ along the half-edges $h'_i \in C_1(\Gamma_i)$.}
    \label{fig:schematic}
\end{figure}

\begin{defn}\label{defn:connected_sum}
For $j = 1,2$, let $\Gamma_j$ be trivalent ribbon graphs with endpoint maps $\e_j: H_j \to E_j$ and $\phi_j: H_j \to F_j$. Given half-edges $h'_j \in H_j$, we define the \emph{sum} $\Gamma_1 \#_{h'_1, h'_2} \Gamma_2$ as the ribbon graph with the following attaching maps.
\begin{itemize}
    \item The map $\phi_1 \# \phi_2: H_1 \cup H_2 \to F_1 \cup F_2$ given by $\phi_1 \# \phi_2(h_j) = \phi_j(h_j)$ for $h_j \in H_j$, and
    \item The map $\e_1 \# \e_2: H_1 \cup H_2 \to E_1 \cup E_2$ given by
    \begin{itemize}
    \item[$\circ$] $\e_1 \# \e_2(h_j) \coloneqq \e_j(h_j)\ \ \text{for }h_j \ne h'_j \in H_j$,
    \item[$\circ$] $\e_1 \# \e_2(h'_1) \coloneqq \e_2(h'_2)$, and
    \item[$\circ$] $\e_1 \# \e_2(h_2') \coloneqq \e_1(h_1')$.
\end{itemize}
\end{itemize}
We write $\Gamma_1 \# \Gamma_2$ if the half-edges $\{h'_1, h'_2\}$ are clear from the context.
If $\iota_j: H_j \to H_j$ are triangle matchings, then we define $\iota_1 \# \iota_2:H_1 \cup H_2 \to H_1 \cup H_2$ by $\iota_1 \# \iota_2(h_j) = \iota_j(h_j)$ for $h_j \in H_j$.
\end{defn}

\begin{ex}[Two squares]\label{ex:two_squares}
If two geometrically triangulated translation surfaces are depicted as polygons in $\R^2$ with opposite sides identified, then the connected sum of their ribbon graphs corresponds to the new translation surface formed by placing those two polygons side by side as in Figure~\ref{fig:summing_two_squares}.
\end{ex}

\begin{figure}
    \centering
    \includegraphics[width=0.5\linewidth]{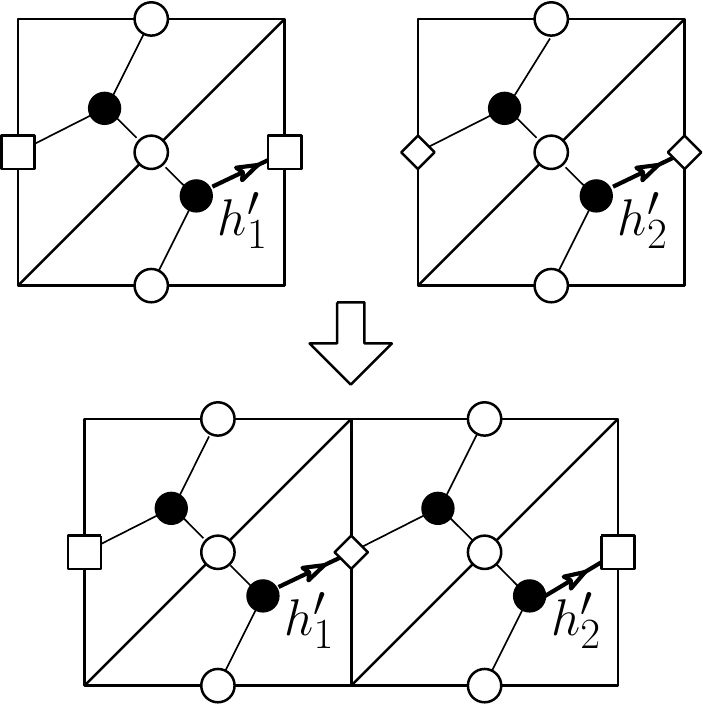}
    \caption{The sum of two square tori yields a torus with two marked points. Note that the endpoints of $h_1'$ and $h_2'$ are depicted by different shapes.}
    \label{fig:summing_two_squares}
\end{figure}

\begin{lem}\label{lem:connected_sum_triangle_matching}
Let $\Gamma_j$ be trivalent ribbon graphs with triangle matchings $\iota_j: H_j \to H_j$, and let $h_j' \in H_j$ be nonseparating edges of the graphs $\Gamma_j$. Then for $\Gamma_1 \#_{h_1', h_2'} \Gamma_2$, the map $\iota_1 \# \iota_2$ is a triangle matching.
\end{lem}

\begin{proof}
For each $h \in H_1 \cup H_2$, let $U_h$ be a small open neighborhood of $h$ in the graph $\Gamma_1 \# \Gamma_2$, and let
\begin{align*}
    U_1 &\coloneqq U_{h'_2} \cup \bigcup_{h \in H_1} U_h,\\
    U_2 &\coloneqq U_{h'_1} \cup \bigcup_{h \in H_2} U_h.
\end{align*}
Retracting the neighborhoods $U_j$ along the contractible sets $U_{h_{j}'}$ and $U_{h_{j+1}'}$ gives homotopy equivalences
\[ U_j \simeq \Gamma_j - h_j'. \]
Moreover, since $U_{h_j'}$ retracts onto $h_j'$, we have a homotopy equivalence
\[ U_1 \cap U_2 \simeq h_1' \cup h_2'. \]
By the Mayer--Vietoris Theorem, the diagram
\begin{center}
\begin{tikzcd}
& U_1 \arrow[rd, "k"] & \\
U_1 \cap U_2 \arrow[ru, "i"] \arrow[rd, "j"] & & \Gamma_1 \# \Gamma_2 \\
& U_2 \arrow[ru, "\ell"] &
\end{tikzcd}
\end{center}
gives a long exact sequence:
\begin{center}
\begin{tikzcd}           
 &
0 \arrow{r} &
H_1(U_1) \oplus H_1(U_2) \arrow{r}{k_* - \ell_*} \arrow{r}{k_* - \ell_*} \arrow[phantom, ""{coordinate, name=Z}]{d} &
H_1(\Gamma_1 \# \Gamma_2)
  \arrow[
    rounded corners,
    to path={
      -- ([xshift=2ex]\tikztostart.east)
      |- (Z) [near end]\tikztonodes
      -| ([xshift=-2ex]\tikztotarget.west)
      -- (\tikztotarget)
    }
  ]{dll}[above, pos=1]{\dee_*} \\
&
H_0(U_1 \cap U_2) \arrow{r}{i_* \oplus j_*} &
H_0(U_1) \oplus H_0(U_2) \arrow{r}{k_* - \ell_*} &
H_0(\Gamma_1 \# \Gamma_2) \arrow{r} & 0.
\end{tikzcd}
\end{center}

Because $h_j'$ is assumed to be nonseparating, there exists a simple cycle $\alpha_j' \in H_1(\Gamma_j)$ of which $h_j'$ is a nontrivial summand.
Regarding $\alpha_j'$ as an element of $C_1(\Gamma_j) = \Z\la H_j \ra$, then because
\[ C_1(\Gamma_1 \# \Gamma_2) = \Z\la H_1 \cup H_2 \ra = \Z\la H_1 \ra \oplus \Z\la H_2 \ra, \]
we may regard $\alpha_1' + \alpha_2'$ as a cycle in $H_1(\Gamma_1 \# \Gamma_2)$.

\begin{claim}
    There is a decomposition
    \[ H_1(\Gamma_1 \# \Gamma_2) \cong \Z\la \alpha_1' + \alpha_2'\ra \oplus H_1(\Gamma_1 - h_1') \oplus H_1(\Gamma_2 - h_2'). \]
\end{claim}
\begin{proof}[Proof of Claim]\let\qed\relax
First, consider $\im(\partial_*)$, which by exactness is $\ker(i_* \oplus j_*)$.
The map $i_* \oplus j_*$, viewed as a map from $\Z^2$ to $\Z^2$, has matrix $\begin{psmallmatrix}
    1 & 1 \\
    1 & 1
\end{psmallmatrix}$,
and its kernel is generated by $f'_1 - f'_2$ where $h'_i = (f'_i, e'_i)$.
We will construct a class in $H_1(\Gamma_1 \# \Gamma_2)$ that maps onto $f'_1 - f'_2$ under $\dee_*$.

Using the definition of the connecting homomorphism, we compute that
\[ \dee_*(\alpha'_1 + \alpha'_2) = \dee \alpha'_1 = f'_1 - e'_1 = f'_1 - f'_2. \]
Therefore $\alpha'_1 + \alpha'_2$ is our desired class mapping onto the image of $\dee_*$.

As a result of the above, we have established a splitting
\[ H_1(\Gamma_1 \# \Gamma_2) = \Z \la \alpha'_1 + \alpha'_2 \ra \oplus \ker(\dee_*). \]
Exactness at $H_1(U_1) \oplus H_1(U_2)$ in the Mayer--Vietoris sequence implies that $k_* - \ell_*$ is injective, so that
\[ \ker(\partial_*) \cong H_1(U_1) \oplus H_1(U_2) \cong H_1(\Gamma_1 - h_1') \oplus H_1(\Gamma_2 - h_2'). \]
This proves the claim.
\end{proof}

We now use the above decomposition to construct the composite triangle matching $\iota_1 \# \iota_2$.
On the one hand, $\iota_1 \# \iota_2$ restricts to $\iota_j$ on the subgroup $H_1(\Gamma_j - h_j')$ of $H_1(\Gamma_1 \# \Gamma_2)$, and $\iota_j$ acts as $-1$ on $H_1(\Gamma_j - h_j')$.
On the other hand, we have $\iota_j(\alpha_j') = -\alpha_j'$ (because $\alpha_j' \in H_1(\Gamma_i)$), so that $\iota_1 \# \iota_2(\alpha_1' + \alpha_2') = -\alpha_1' - \alpha_2'$.
We conclude that $\iota_1 \# \iota_2$ acts as $-1$ on all of $H_1(\Gamma_1 \# \Gamma_2)$.
\end{proof}

Finally, we use the following lemma in order to apply Lemma~\ref{lem:connected_sum_triangle_matching} to geometric triangulations of translation surfaces.

\begin{lem}\label{lem:nonseparating}
If $\Gamma$ is the dual graph to a geometric triangulation $\tau$ of a translation surface $(X, \omega)$, then every half-edge $h \in H$ is nonseparating.
\end{lem}

\begin{proof}
First observe that the half-edge $h  = (f, e)$ is separating if and only if $\Gamma \setminus \{e\}$ is disconnected, because $h$ retracts onto its endpoint $e$. The preimage of the vertex $e$ of $\Gamma$ under the deformation retraction $S_g \setminus Z \twoheadrightarrow \Gamma$ is the interior of the edge $e$ of the triangulation $\tau$. Therefore, $\Gamma \setminus \{e\}$ is disconnected if and only if $(S_g \setminus Z) \setminus e$ is disconnected, which means that the topological arc $e$ bounds a subsurface of $S_g$. Since the triangulation $\tau$ is geometric, the edge $e$ is a saddle connection on $(X, \omega)$, and hence its period $\int_e \omega$ is nonzero. An arc that bounds a subsurface is nulhomologous, and hence must have trivial period $\int_e \omega$. We conclude that if $e$ is an edge of a geometric triangulation $\tau$, then $h = (f, e)$ cannot be separating.
\end{proof}

\subsection{Arboreal origamis}\label{subsec:arboreal}
In this subsection, we give combinatorial criteria that determine whether the \emph{standard triangulation} $\tau(h, v)$ of an \emph{origami} $X(h, v)$ admits a triangle matching (Theorem~\ref{thm:arboreal_origamis}).
Our proof uses Lemmas~\ref{lem:connected_sum_triangle_matching}~and~\ref{lem:nonseparating} to establish that $\tau(h, v)$ admits a triangle matching if and only if the core curves of $X(h, v)$ form an \emph{arboreal network} as in \cite[Definition 4.6]{CalderonSalter}.
Hamenst\"adt \cite[Definition 1.2]{Hamenstadt} independently employed arboreal networks under the name \emph{admissible curve systems}.
We begin by recalling some basic definitions.

\begin{defn}\label{defn:origami}
For $s \in \N$, let $\Sym(s)$ denote the symmetric group on $s$ symbols.
Given a pair of permutations $(h, v) \in \Sym(s)^2$ such that $h$ and $v$ act transitively on $[s] \coloneqq \{1,\dots,s\}$, we take the set $[0,1]^2 \times [s]$ and impose the relations
\[ (1,y, j) \sim (0,y, h(j))\text{ and }(x, 1, j) \sim (x, 0, v(j)). \]
The quotient $X \coloneqq ([0,1]^2 \times [s])/\sim$ inherits an abelian differential $\omega$ from the differential $dz$ on $[0,1]^2$. The image of every point of the form $(0, 0, j) \in [0,1]^2 \times [s]$ at which $\omega$ does not vanish is regarded as a \emph{marked point} (cf. Definition~\ref{defn:translation_surface}) of $\omega$.

We denote by $X(h, v) \coloneqq (X, \omega)$ the translation surface structure determined by this construction, and we say that $X(h, v)$ is an \emph{origami}.
Because $h$ and $v$ act transitively on $[s]$, the resulting surface is connected.
For each $j \in [s]$, we say that the image of $[0,1] \times \{j\}$ in $X(h, v)$ under the quotient mapping is a \emph{square} of $X(h, v)$.
\end{defn}

\begin{defn}
We say that a simple closed curve $\beta \subset X(h, v)$ is a \emph{horizontal (resp. vertical) cylinder core curve} if every point $z \in \beta$ is of the form $(x, \tfrac12, j)$ (resp. $(\tfrac12, y, j)$).
Let $\mathcal N(h,v)$ denote the collection of all cylinder core curves on $X(h, v)$, called the \emph{network} for $X(h, v)$.
\end{defn}

\begin{defn}\label{defn:standard_triangulation}
Let $X(h, v)$ be an origami. The \emph{standard triangulation} $\tau(h, v)$ of $X(h, v)$ is the geometric triangulation whose faces are the isosceles right triangles with hypotenuses equal to the positively sloped diagonals of the squares of $X(h, v)$.
\end{defn}
See Figure~\ref{fig:square_l_combinatorial} for the example of the square L.

\begin{defn}\label{defn:geometrically_simple}
Let $X(h, v)$ be an origami.
If the geometric intersection number $i(\eta, \eta')$ is at most 1 for each pair of core curves $\eta, \eta' \in \mathcal N(h, v)$, we say that $X(h,v)$ is \emph{geometrically simple}.
\end{defn}

\begin{defn}\label{defn:arboreal}
If $X(h,v)$ is geometrically simple, let $\Lambda_{\mathcal N(h, v)}$ denote the simple graph with vertices $\mathcal N(h, v)$, where we add an edge between $\eta, \eta' \in \mathcal N(h, v)$ if and only if $\eta$ and $\eta'$ intersect.
We say that the surface $X(h, v)$ and its network $\mathcal N(h, v)$ are \emph{arboreal} if $\Lambda_{\mathcal N(h, v)}$ is a tree.
\end{defn}

The square L admits an arboreal network.
We give another example in Example~\ref{ex:escher}.

\begin{lem}\label{lem:same_cyls}
    If $X(h, v)$ is a geometrically simple orgami whose standard triangulation admits a triangle matching $\iota: H \to H$, then for all $h \in H$ the half-edges $h$ and $\iota(h)$ must reside in the same square.
\end{lem}
\begin{proof}
Let $h = (f, e) \in H$, and let $(f', e') = \iota(h)$.
Recall the map $r: H \to H_1(S_g, Z; \Z)$ from Definition~\ref{defn:relative_homology_map}.
Assume that $r(f,e)$ is a diagonal edge oriented so that it points down and to the left --- the other cases will follow from the fact that a triangle matching commutes with the $\mathbb{Z}/3\mathbb{Z}$ action on $H_f$.

Let $\alpha$ and $\beta$ denote the core curves of the horizontal and vertical cylinders, respectively, that contain the face $f$.
By Lemma~\ref{lem:Sam's_favorite_formula} we have
\[ 1 = \langle \alpha, r(f, e) \rangle = -\langle \alpha, r(f', e') \rangle, \]
so that $f$ and $f'$ are contained in the same horizontal cylinder.
Moreover,
\[ 1 = \langle \beta, r(f, e) \rangle = -\langle \beta, r(f', e') \rangle, \]
so that $f$ and $f'$ are contained in the same vertical cylinder.
Because $X(h, v)$ is geometrically simple, $\alpha$ and $\beta$ intersect only once, and hence their respective cylinders intersect in a unique square.
It follows that $r(f', e') = -r(f, e)$; that is, the half-edges $h$ and $\iota(h)$ reside in the same square, as desired.
\end{proof}

\begin{thm}\label{thm:arboreal_origamis}
Let $X(h, v)$ be a geometrically simple origami with $h, v \in \Sym(s)$.
The following are equivalent.
\begin{enumerate}
\item $X(h, v)$ is arboreal,
\item $\mathcal \tau(h, v)$ admits a triangle matching,
\item $\#\mathrm{cycles}(h) + \#\mathrm{cycles}(v) = s + 1$,
\end{enumerate}
where $\# \mathrm{cycles}(\sigma)$ denotes the number of cycles of a permutation $\sigma \in \Sym(s)$.
\end{thm}

\begin{proof}
\textbf{(i) $\implies$ (ii):}
We proceed by induction on $s$.
The case $s = 1$ is immediate because the square torus is hyperelliptic (indeed, elliptic) and thus its standard triangulation $\tau_0$ admits a triangle matching by Lemma~\ref{lem:hyperelliptic_delaunay}.
Now, let $s > 1$, and suppose that the claim holds for all origamis with fewer than $s$ squares.

Let $X(h, v)$ be an $s$-square origami.
By hypothesis, the graph $\Lambda_{\mathcal{N}(h, v)}$ is a tree and therefore has a leaf that corresponds to a cylinder core curve $\alpha$.
Without loss of generality, assume that $\alpha$ is horizontal.
Because $\alpha$ is a leaf of $\Lambda_{\mathcal{N}(h, v)}$, $\alpha$ intersects only one vertical core curve $\beta$.
Therefore, $\alpha$ is contained in a single square.
Without loss of generality, this is the $s$th square.
The graph $\Lambda_{\mathcal{N}(h, v)} \setminus \{\alpha\}$ corresponds to the origami $X(h', v')$, where $h', v' \in \Sym(s - 1)$ have the index ``$s$'' removed.
In particular, $\Lambda_{\mathcal N_{(h', v')}} = \Lambda_{\mathcal{N}(h, v)} \setminus \{\alpha\}$ is also a connected tree, since $\alpha$ was a leaf.

By the inductive hypothesis, the standard triangulation $\tau(h', v')$ admits a triangle matching.
Moreover, $\tau(h, v)$ is the connected sum $\tau(h', v') \# \tau_0$, as in Definition~\ref{defn:connected_sum}.
Therefore, $\tau(h, v)$ admits a triangle matching by Lemmas~\ref{lem:connected_sum_triangle_matching}~and~\ref{lem:nonseparating}.\\

\noindent \textbf{(ii) $\implies$ (i):}
Suppose that $X(h, v)$ is not arboreal.
Then, the graph $\Lambda_{\mathcal{N}(h, v)}$ has a simple cycle $\alpha_1, \beta_1, \dots, \beta_{n-1}, \alpha_1$.
Because $X(h, v)$ is geometrically simple, each consecutive pair of core curves meet at a unique point.
Connecting these intersection points with appropriate arcs of core curves, we produce a simple cycle $\eta$ in $\Gamma$ that changes direction at each intersection point.

We now claim that $\tau(h, v)$ cannot admit a triangle matching.
Indeed, Lemma~\ref{lem:same_cyls} implies that two matched half-edges must reside in the same square.
But after a homotopy, $\eta$ will intersect some class $r(f, e)$ in a square containing an intersection point while avoiding the class $r(\iota(f, e))$ in that same square.
This violates Lemma~\ref{lem:Sam's_favorite_formula} which implies $\langle \eta, r(f, e) \rangle = - \langle \eta, r(\iota(f, e)) \rangle$.\\

\noindent \textbf{(i) $\iff$ (iii):}
We have that $\Lambda_{\mathcal{N}(h, v)}$ is a tree if and only if it has one more vertex than its number of edges.
Note that $\Lambda_{\mathcal N(h, v)}$ has $\# \textrm{cycles}(h) + \# \textrm{cycles}(v)$ vertices corresponding to horizontal and vertical cylinders and $s$ edges corresponding to intersection points of horizontal and vertical cylinders.
We conclude that $X(h, v)$ is arboreal if and only if $\# \textrm{cycles}(h) + \# \textrm{cycles}(v) = s + 1$.
\end{proof}

\begin{cor}\label{cor:arboreal_in_every_stratum}
For $g \ge 3$, every nonhyperelliptic component of the projectivized stratum of $\P \Omega \mathcal T_{g, n}(\kappa)$ contains a contractible iso-Delaunay region.
\end{cor}
\begin{proof}
Lemma~6.15 of \cite{CalderonSalter}, or equivalently Theorem~1.3 of \cite{Hamenstadt}, provides explicit arboreal networks $\mathcal N$ that, after applying the Thurston--Veech construction (see, for instance, \cite[\S 6]{McMullenSurvey}), produce abelian differentials in every nonhyperelliptic component of the moduli stratum $\Omega \mathcal M_{g, n}(\kappa)$.
Replacing all the rectangles in the Thurston--Veech decomposition with unit squares (and choosing appropriate mapping class group lifts) produces origamis $X(h, v)$ in each nonhyperelliptic component of $\Omega \mathcal T_{g, n}(\kappa)$.

Note that the matrix
\[ A =
\begin{pmatrix}
1 & -1/2\\
0 & \sqrt{3}/2\\
\end{pmatrix}
\]
maps the triangles with vertices $\{(0,0), (1, 0), (0, 1)\}$ and $\{(1,0), (1,1), (0,1)\}$ onto equilateral triangles.
The matrix $A$ then acts on $X(h, v)$ to give a translation surface $A \cdot X(h, v)$ on which every triangle of $\tau(h, v)$ is equilateral.
An equilateral triangulation is nondegenerate Delaunay, so the projective class of $A \cdot X(h, v)$ belongs to the nonempty iso-Delaunay region $\Delta_{\mathcal D}(\tau(h, v))$.
Since $\mathcal{N}$ is aboreal, Theorem~\ref{thm:arboreal_origamis} implies that $\tau(h, v)$ admits a triangle matching and $\Delta_{\mathcal D}(\tau(h, v))$ is contractible by Theorem~\ref{thm:main_theorem_precise}.
\end{proof}

We conclude this subsection with two contrasting examples of origamis.

\begin{ex}[Prym Origami]
Let $h = (12)(3)(45)$ and $v = (1)(234)(5)$.
The origami $X(h, v)$ resides in the stratum $\Omega \mathcal{T}_{3, 1}(4)$, and Theorem~\ref{thm:arboreal_origamis} implies that $X(h, v)$ is arboreal.
Recall that the triangle matching $\iota$ identifies together the half-edges in each square as in the case of the square L surface in Example~\ref{ex:square_l_matching}.

In fact, $X(h, v) = (X, \omega)$ is a minimal Prym eigenform and hence possesses a holomorphic Prym involution $\iota_\textrm{Prym} : X \to X$ that acts as $-1$ on the underlying 1-form $\omega$.
(See McMullen~\cite{McmullenPrym} for more details.)
But while $\iota_\textrm{Prym}$ induces a ribbon graph automorphism, it is \emph{not} a triangle matching --- it only acts as $-1$ on a factor of $H_1(X)$.
For instance, two horizontal core curves are exchanged instead of being fixed setwise.
\end{ex}

\begin{ex}[Escher Staircase]\label{ex:escher}
Let $h = (12)(34)(56)$, $v = (23)(45)(16)$ and consider the origami $X(h, v)$ as in Figure~\ref{fig:escalator}.

\begin{figure}
    \centering    \includegraphics[width=0.7\linewidth]{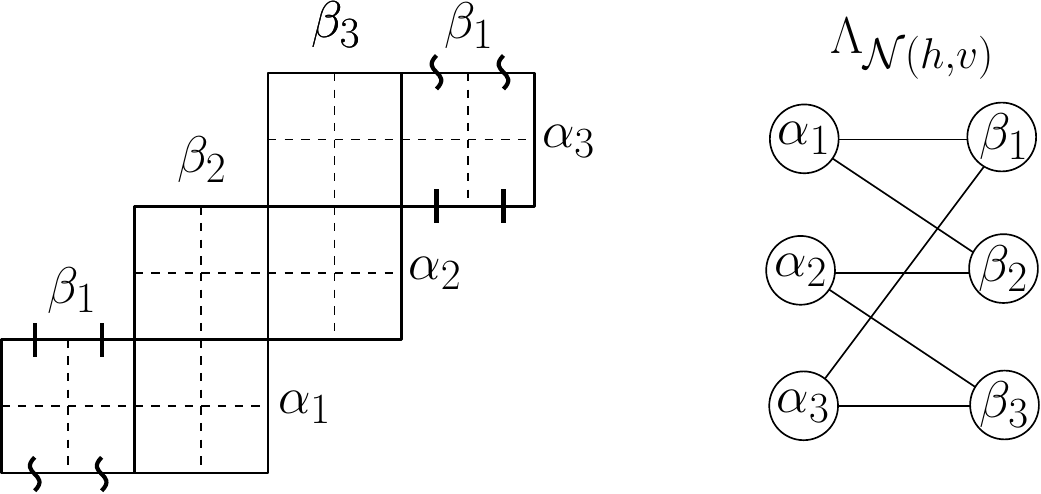}
    \caption{The origami $X(h, v)$ with $h = (12)(34)(56)$ and $v = (23)(45)(16)$.
    All side identifications are implied except for those with edge markings.}
    \label{fig:escalator}
\end{figure}

As $\# \textrm{cycles}(h) + \# \textrm{cycles}(v) = 6$, but $s+1 = 7$, Theorem~\ref{thm:arboreal_origamis} implies that $\tau(h, v)$ does \textit{not} admit a triangle matching.
It would be very interesting to determine whether $\Delta_{\mathcal D}(\tau(h, v))$ is convex in the appropriate angle space.
\end{ex}

\section{Future directions}\label{sec:future_directions}
The structure of the iso-Delaunay regions is of independent interest, as seen in the works \cite{MasurSmillie}, \cite{Bowman}, \cite{VeechBicuspid}, \cite{VeechDelaunay}, and \cite{Rivin} mentioned in the introduction.
In addition, if it were known that \emph{all} the iso-Delaunay regions were convex, then these regions could be used to form a cellulation of $\Omega \mathcal T_{g, n}(\kappa)$, and also of the moduli stratum $\Omega \mathcal M_{g, n}(\kappa)$.
This is precisely what is done in the case of iso-Delaunay regions for the $L^\infty$-metric in \cite{Zykoski}.

In that context, the induced cellulation may be understood computationally by way of linear programming, but precise analysis of the general combinatorial properties of that cellulation are stymied by its sensitive dependance on the geometry of translation surfaces.
However, the main results of this paper depend almost entirely on the combinatorial structure of surface triangulations.
This gives us optimism that these iso-Delaunay regions may produce a cellulation of $\Omega \mathcal M_{g, n}(\kappa)$ about which one can more easily prove structural results, and hence obtain results on the topology of $\Omega \mathcal M_{g, n}(\kappa)$, a subject about which few general results are known.

\newpage
\printbibliography
\end{document}